\documentclass[13pt]{amsart}
\usepackage{amsmath}
\usepackage{tikz}
\usepackage{amsfonts}
\usepackage{amsfonts,latexsym,rawfonts,amsmath,amssymb,amsthm}
\usepackage[plainpages=false]{hyperref}

\usepackage{pdfsync}

\usepackage[a4paper]{geometry}
\usepackage{esint}
\usepackage{graphics,graphicx}
\usepackage{tikz}

\numberwithin{equation}{section}

\newcommand{\beq}{\begin{equation}}
\newcommand{\eeq}{\end{equation}}
\newcommand{\beqs}{\begin{eqnarray*}}
\newcommand{\eeqs}{\end{eqnarray*}}
\newcommand{\beqn}{\begin{eqnarray}}
\newcommand{\eeqn}{\end{eqnarray}}
\newcommand{\beqa}{\begin{array}}
\newcommand{\eeqa}{\end{array}}

\def\lra{\longrightarrow}

\def\p{\prime}

\def\bc{\begin{center}}
\def\ec{\end{center}}

\def\p{\partial}

\def\begeq{\begin{equation}}
\def\endeq{\end{equation}}
\def\and{\quad{\rm and}\quad}

\let\lra=\longrightarrow

\def\mapright\#1{\,\smash{\mathop{\lra}\limits^{\#1}}\,}

\newtheorem{prop}{Proposition}[section]
\newtheorem{theo}[prop]{Theorem}
\newtheorem{lem}[prop]{Lemma}

\newtheorem{cor}[prop]{Corollary}
\newtheorem{rem}[prop]{Remark}
\newtheorem{exa}[prop]{Example}
\newtheorem{defi}[prop]{Definition}

\begin{document}

\title{K-stability and polystable degenerations of polarized spherical  varieties}
\author{Yan Li$^{*}$ and Bin Zhou$^{\dag}$}

\address{$^{*}$School of Mathematics and Statistics, Beijing Institute of Technology, Beijing, 100081, China.}
\address{$^{\dag}$School of Mathematical Sciences, Peking
University, Beijing 100871, China.}
\email{liyan.kitai@yandex.ru\,,\ bzhou@pku.edu.cn}

\thanks {$^*$Partially supported by NSFC Grant 12101043 and Beijing Institute of Technology Research Fund Program for Young Scholars, No. 3170012222012.}
\thanks {$^{\dag}$ Partially supported by National Key R$\&$D Program of China SQ2020YFA0712800 and NSFC  Grant 11822101.}
\subjclass[2020]{Primary: 14M27; Secondary: 32Q26, 53C55.}

\keywords{K\"ahler manifolds, K-stability, spherical variety.}

\begin{abstract}
In this paper, we study the K-stability of polarized spherical varieties. After reduction, it can be treated as a variational problem of the reduced functional of the Futaki invariant on the associated moment polytope. With the convexity constraint of the problem, the minimizers are shown to satisfy the homogeneous Monge-Amp\`ere equation (HMA). When the spherical variety has rank two, a simpler characterization can be established through properties of the HMA.
As an application, we determine the strict semistability and polystable degenerations for Fano spherical varieties of rank two.
\end{abstract}

\maketitle


\section{Introduction}

In K\"ahler geometry, a fundamental problem is to find algebro-geometric conditions for the existence of special metrics on K\"ahler manifolds, including K\"ahler-Einstein metrics, constant scalar curvature metrics, extremal metrics, K\"ahler-Ricci solitons, etc.
One of the most widely studied conditions is the K-stability, which concerns on the positivity of the Futaki invariant for test configurations (one parameter degenerations with an $\mathbb C^*$-action) of the manifold \cite{Ti, Do}.
Generally, testing the K-stability is an infinite dimensional problem and is very difficult.
In a series of papers \cite{Do, D2, D3}, Donaldson established a beautiful reduction of the K-stability of toric manifolds, and proved the existence of constant scalar curvature metrics on K-stable polarized toric
surfaces. Later in an expository paper \cite{Do08}, some of the intuitions are extended to \emph{multiplicity-free} manifolds, or \emph{spherical varieties} in notation of algebraic geometry.


The class of spherical varieties is a huge class of varieties with large symmetry. It contains many important classes such as generalized flag manifolds, toric varieties, homogenous toric bundles, group compactifications, etc. The K\"ahler geometry on these classes has been explored extensively (see \cite{AK, Del1, Del4, PS10, Ra, Timashev-Sbo}). In recent years, many highly non-trivial examples with interesting properties have been found in this class, including a class of K-unstable Fano spherical manifolds \cite{Del3} and the first examples of Type-II solution of the K\"ahler-Ricci flow on Fano manifolds \cite{LTZ}.
Such examples do not exist in toric varieties.  On one hand, by the famous Luna-Vust theory \cite{Luna-Vust, Timashev-book}, a spherical variety admits a nice combinatory data which encodes its geometric information. One can derive certain combinatory criterions for geometric problems. On the other hand, the class of spherical varieties has a nice ``closed" property:  an invariant divisor is also a spherical variety. In addition, the equivariant degeneration of a spherical variety equipped with a complex torus $(\mathbb C^*)^r$-action (that is, an $\mathbb R$-test configuration), as well as its limit (if normal), are also spherical varieties. This was studied in \cite{Popov-1986,Del3} when $r=1$ (that is, a ($\mathbb Z$-)test configuration). We remark that \cite{Popov-1986} studied general equivariant filtrations of spherical varieties. Later in \cite{LL-arXiv-2021}, Li-Li studied the general equivariant $\mathbb R$-test configurations and found the limit of K\"ahler-Ricci flow on group compactifications. In this paper, we investigate the K-stability of polarized spherical varieties.

To state our results, we first introduce some notations. Let $G$ be a complex, connected reductive Lie group and $B$ a fixed Borel subgroup in it. A homogenous space $G/H$ for some subgroup $H\subset G$ is called \emph{spherical} if it contains an open and dense $B$-orbit. By \cite{Luna-Vust, Timashev-book}, $G/H$ determines a valuation cone $\mathcal V$ which is a fundamental chamber of the associated little group $\mathcal W$. In fact, $\mathcal W$ is the Weyl group of the spherical root system $\Phi$. Following \cite{Timashev-book}, we fix a system of positive roots $\Phi_+\in\Phi$ so that $\mathcal V$ is the anti-dominate Weyl chamber. A $G$-equivariant, polarized normal compactification $(M,L)$ of $G/H$ consists of a variety $M$ admitting $G$-action with an open dense orbit isomorphic to $G/H$ and a $G$-linearized ample line bundle $L$. For any such $(M,L)$, there is a moment polytope $P_+$ that encodes the structure of representations of $G$ on the space of holomorphic sections of tensor powers of $L$. We refer  the readers to Section 2.1 below for more details.

In \cite{Del-202009}, Delcroix obtained an expression of the K-stability  of  polarized spherical varieties in terms of the combinatorial data, which generalizes the case of toric varieties by Donaldson. As can be seen in Section 2.3 below, we can generalize this reduction to the relative Futaki invariant. We will reduce it to a linear functional $\mathcal L_X(\cdot)$ (see \eqref{L(u)} below) on a class $\mathcal C_{1,\mathcal W}$ of certain convex functions, which corresponds to the modula of equivariant test configurations of $(M,L)$. Hence, the relative K-stability can be described by the sign of $\mathcal L_X(\cdot)$. Then we have:

\begin{theo}\label{Main-thm1}
Let $G/H$ be a spherical homogenous space and $\mathcal W$ the corresponding little Weyl group. Let $(M,L)$ be a $G$-equivariant, polarized normal compactification of $G/H$ and $P_+$ be its associated polytope. Assume that $P_+=P\cap \mathcal V_+$ for a convex $\mathcal W$-invariant polytope $P$. Suppose that $\mathcal L_X(\cdot)\geq0$ for any function in $\mathcal C_{1,\mathcal W}$ and there is a non-central affine function $u_0\in\mathcal C_{1,\mathcal W}$ so that $\mathcal L_X(u_0)=0$.
Then:
\begin{enumerate}
\item $u_0$ is a  generalized solution of the homogenous Monge-Amp\`ere equation on $P$
in the sense of Alexandrov;
\item  If $rank(G/H)=2$, there exists a $\mathcal W$-invariant simple piecewise linear function $\bar u\in\mathcal C_{1,\mathcal W}$, which is not a central affine function, so that $\mathcal L_X(\bar u)=0$.
\end{enumerate}
\end{theo}

\begin{rem}
Theorem \ref{Main-thm1} (2) was first proved by Donaldson for toric varieties \cite{Do} with the assumption $\bar S+\theta_X>0$. The assumption was removed in \cite{Wang-Zhou} with the idea of Theorem \ref{Main-thm1} (1).
\end{rem}


Back to the  K-stability, an explicit criterion for the K-stability of Fano spherical varieties has been obtained by \cite{Del3}.
Theorem \ref{Main-thm1} can be applied to Fano spherical varieties to study the strictly K-semistabity and their polystable degenerations. The polystable degeneration of a $\mathbb Q$-Fano variety was introduced in \cite{Li-Wang-Xu} by an abstract construction. \cite[Theorem 1.3]{Han-Li} showed the uniqueness of the polystable degeneration.
Recently, it is proved in \cite{Chen-Sun-Wang,Han-Li} that the limit space  of the K\"ahler-Ricci flow can be derived from the initial one via at most two $\mathbb R$-test comfigurations: the ``semistable degeneration" and ``polystable degeneration".  Here we have:
\begin{theo}\label{Main-thm2}
Let $G/H$ be a spherical homogenous space. Consider its $\mathbb Q$-Fano compactifications.
\begin{enumerate}
\item If the spherical root system $\Phi$ is irreducible and the valuation cone has trivial central part $\mathcal V_z=\{O\}$, there is no strictly K-semistable $\mathbb Q$-Fano compactification of $G/H$ whose moment polytope $P_+$ extends to a $\mathcal W$-invariant convex polytope $P\subset\mathfrak M_\mathbb R$;

\item Assume $rank(G/H)=2$ and  $\dim(\mathcal V_z)=1$. If $M$ is a strictly K-semistable $\mathbb Q$-Fano compactification of $G/H$ whose moment polytope $P_+$ extends to a $\mathcal W$-invariant convex polytope $P\subset\mathfrak M_\mathbb R$. Then there is a unique polystable degeneration that degenerates $M$ to a $\mathbb Q$-Fano horospherical variety, which can be explicitly determined by its moment polytope.
\end{enumerate}
\end{theo}

The above theorem also gives a new proof of the K-semistable criterion in \cite[Theorem 5.3]{Del3} under the assumption that rank$(G/H)=2$ (see Remark \ref{recover} below). The proof in \cite{Del3} uses a deep theorem of \cite{Li-Xu-2014}. Our approach only uses the property of solutions to the homogenous Monge-Amp\`ere equation and elementary analysis. Note that a $\mathbb Q$-Fano horospherical variety is always modified K-stable (see Remark \ref{mod} for the definition).
The second item confirms that if a strictly K-semistable $G/H$-compactification exists, then its ``polystable degeneration" is uniquely determined. An explicit description for this degeneration will be given in the proof (see Proposition \ref{limit-kss-Q-fano} below).
In particular, when $M$ is a group compactification of rank two, we have a more complete result.

\begin{theo}\label{Main-thm3}
Let $\mathcal G$ be a connected, complex reductive group of rank two and $Z(\mathcal G)$ be its center.  Consider its $\mathbb Q$-Fano compactifications.
\begin{enumerate}
\item If $\mathcal G$ is semisimple or toric, there is no strictly-K-semistable $\mathbb Q$-Fano compactifications;
\item If $\dim(Z(\mathcal G))=1$ and $M$ is a strictly K-semistable $\mathbb Q$-Fano $\mathcal G$-compactification.
Then there is a unique polystable degeneration that degenerates $M$ to a $\mathbb Q$-Fano horospherical variety, which can be explicitly determined by its moment polytope.
\end{enumerate}
\end{theo}

The structure of the paper is as follows: Section \ref{pre} is devoted to a review on notations and conventions on polarized spherical varieties and the K-stability. In Section \ref{mini}, we establish a characterization on minimizers of general convex variational problem by invoking the idea in \cite{Wang-Zhou} and apply it to prove Theorem \ref{Main-thm1} (1).
The further simplification (see Theorem \ref{Main-thm1} (2)) on the (relative) K-stability of spherical varieties of rank two will be investigated in Section \ref{r2}. Theorem \ref{Main-thm2} and \ref{Main-thm2} will be proved in Section \ref{ap}. In all these results, we assume $P_+=P\cap \mathcal V_+$ for a convex $\mathcal W$-invariant polytope $P$. We present a proof in the Appendix  that this condition is  automatically satisfied for group compactifications.

\vskip 20pt

\section{Stability of polarized spherical varieties}\label{pre}

\subsection{Spherical varieties}
In the following we review the general facts about spherical varieties. The original results go back to \cite{Luna-Vust}. We use \cite{Timashev-survey,Timashev-book} as the main reference.

Let $G$ be a complex, connected reductive group and $B$ a fixed Borel subgroup. A closed subgroup $H$ is called a \emph{spherical subgroup} if there is an open and dense $B$-orbit in $G/H$.\footnote{In fact, it suffice to assume only $B$ has an open orbit in $G/H$. The denseness can be derived from a Theorem of Rosenlicht, cf. \cite[Theorem 3.1]{Timashev-survey} and \cite[Section 25]{Timashev-book}.} In this case we call $G/H$ a \emph{spherical homogenous space}. A projective normal variety $M$ is called a \emph{$G$-spherical variety} if $M$ admits a holomorphic $G$-action so that there is an open dense orbit which is $G$-equivariantly isomorphic to some spherical homogenous space $G/H$ (cf. \cite[Definition 3.2]{Timashev-survey}). In addition, if $M$ admits an ample $G$-linearized line bundle $L$, we call $(M,L)$ a \emph{polarized $G$-spherical variety}.


Let $H$ be a spherical subgroup of $G$ with respect to the Borel subgroup $B$. Assume that the action of $G$ on the function field $\mathbb C(G/H)$ of $G/H$ is given by
$$(g^*f)(x):=f(g^{-1}\cdot x),~\forall g\in G, x\in \mathbb C(G/H)~\text{and}~f\in\mathbb C(G/H).$$
A function $f(\not=0)\in\mathbb C(G/H)$ is called \emph{$B$-semiinvariant} if there is a character of $B$, denoted by $\chi$ so that $b^*f=\chi(b)f$ for any $b\in B$. Since there is an open dense $B$-orbit in $G/H$, any two $B$-semiinvariant functions of a same character can differ from each other only by a scalar multiple.

Let $\mathfrak M$ be the lattice of $B$-characters of $B$-semiinvariant functions. We call its rank the \emph{rank of $G/H$} (also \emph{rank of $M$}). Let $\mathfrak N={\rm Hom}_\mathbb Z(\mathfrak M,\mathbb Z)$ be its $\mathbb Z$-dual. There is a map $\varrho$ which maps a valuation $\nu$ of $\mathbb C(G/H)$ to an element $\varrho(\nu)$ in $\mathfrak N_\mathbb Q=\mathfrak N\otimes_\mathbb Z\mathbb Q$ by
\begin{align}\label{varrho}
\varrho(\nu)(\chi)=\nu(f),
\end{align}
where $f$ is a $B$-semiinvariant function of $\chi$. This is well-defined since $f$ is unique up to multiplying a non-zero constant. A fundamental result states that $\varrho$ is injective on $G$-invariant valuations and the image is a convex cone $\mathcal V(G/H)$ in $\mathfrak N_\mathbb  Q$, called the \emph{valuation cone} of $G/H$ (cf. \cite[Section 19]{Timashev-book}). Moreover, $\mathcal V$ is a solid cosimplicial cone which is a (closed) fundamental chamber of a crystallographic reflection group, called the {\it little Weyl group} (denoted by $\mathcal W$; cf. \cite[Sections 22]{Timashev-book}). In fact, $\mathcal W$ is the Weyl group of the \emph{spherical root system} $\Phi$ of $G/H$ (cf. \cite[Section 30]{Timashev-book}). For convenience, we usually choose a system of positive roots $\Phi_+\subset\Phi$ so that $\mathcal V$ is the anti-dominate Weyl chamber (cf. \cite[Section 24]{Timashev-book}). Note that the fixed point set $\mathcal V_z$ of the $\mathcal W$-action is the maximal linear space contained in $\mathcal V$. It is known that $\mathfrak{z(g)}\subset\mathcal V_z$. In the following, we will write the dominate chamber $\mathcal V_+:=-\mathcal V$ for short.

\begin{exa}\label{grp-exa}
Let $\mathcal G$ be a connected, complex reductive group. Let $\mathcal B^+\subset \mathcal G$ be a Borel subgroup of $\mathcal G$ and $\mathcal B^-$ be its opposite group. Take $G=\mathcal G\times \mathcal G$, $H={\rm diag}(\mathcal G)$ and $B=\mathcal B^-\times \mathcal B^+$. Then by the well-known Bruhat decomposition, $H$ is a spherical subgroup. Hence a $\mathcal G$-compactification is a $G$-spherical variety. By taking an involution $\sigma(g_1,g_2)=(g_2,g_1)$ for any $g_1, g_2\in \mathcal G$ on $G$, we see that $H=G^\sigma$. Thus $G/H$ is in addition a symmetric space.

Let us fix the maximal complex torus $\mathcal T^\mathbb C=\mathcal B^+\cap \mathcal B^-$ of $\mathcal G$. Denote by $ \Phi^\mathcal G$ the root system of $(\mathcal G, \mathcal T^\mathbb C)$ and $\Phi_+^\mathcal G\subset\Phi^\mathcal G$ the positive roots with respect to $\mathcal B^+$. By a direct computation we can identify $\mathfrak M$ with the lattice of weights $\mathfrak M^{\mathcal G}$ of $\mathcal G$ via the anti-diagonal embedding:
\begin{align*}
\mathfrak M^{\mathcal G}&\to\mathfrak M\subset\mathfrak M^{\mathcal G\times\mathcal G}\\
\chi&\to(\chi,-\chi).
\end{align*}
Here $\mathfrak M$ is considered as a sublattice of the lattice $\mathfrak M^{\mathcal G\times\mathcal G}$ of $\mathcal G\times\mathcal G$. Further calculation shows that under this identification the spherical root system $\Phi$ is identified with $\Phi^\mathcal G$. Consequently, the little Weyl group $\mathcal W$ is isomorphic to the Weyl group of $(\mathcal G,\mathcal T^\mathbb C)$. Hence we identify $\mathcal V$ with the anti-dominate (closed) Weyl chamber $-\mathfrak a_+$ of $\Phi_+^\mathcal G$ in $\mathfrak a:=\mathfrak M_\mathbb R^\mathcal G$ (cf. \cite[Sections 4-8]{Timashev-Sbo}), and $\mathcal V_z$ is the centre $\mathfrak{z(g)}$ of the Lie algebra $\mathfrak g$ of $\mathcal G$. 
\end{exa}

\subsection{Moment polytope of a polarized spherical variety}
Let $(M,L)$ be a polarized $G$-spherical variety. Fix a maximal compact torus $T$ of $G$ so that its complexification $T^\mathbb C$ is the maximal complex torus of $G$ that lies in $B$. Let $ \Phi^G$ be the root system of $(G,T^\mathbb C)$ and $ \Phi^G_+$ be the positive roots with respect to $B$. Then for any $k\in\mathbb N$ we can decompose $H^0(M,L^k)$ as a direct sum of irreducible $G$-representations:
\begin{align}\label{H0kL}
H^0(M,L^k)=\bigoplus_{\lambda\in P_{+,k}} V_{\lambda},
\end{align}
where $P_{+,k}$ is a finite set of dominate $B$-weights and $ V_{\lambda}$ the irreducible representation of $G$ with highest weight $\lambda$. Set
$$\displaystyle P_+:=\overline{\bigcup_{k\in\mathbb N}(\frac1kP_{+,k})}.$$
Then $P_+$ is indeed a polytope in $\mathfrak M_\mathbb R$ (cf. \cite[Section 3.2]{Brion89}). We call it the \emph{moment polytope of $(M,L)$}. Clearly, the moment polytope of $(M,L^k)$ is $k$-times the moment polytope of $(M,L)$ for any $k\in\mathbb N_+$.

Let $\mathfrak t$, $\mathfrak g$ be the Lie algebras of $T$, $G$, respectively.  Let $J$ be the complex structure of $G$. Fix an inner product $\langle\cdot,\cdot\rangle$ on the Lie algebra $J\mathfrak t$, which extends the Killing form on $J\mathfrak t\cap[\mathfrak g,\mathfrak g]$. By the Weyl character formula \cite[Section 3.4.4]{Zhelobenko-Shtern},
$$\dim( V_{\lambda})=\prod_{\alpha\in\Phi^G_+}\frac{\langle \lambda+2\rho,\alpha\rangle}{\langle\alpha,2\rho\rangle},$$
where $$\rho=\frac12\sum_{\alpha\in\Phi^G_+}\alpha.$$
Combining with \eqref{H0kL} and the Riemann-Roch formula \cite[Section 1.4]{Pukhlikov-Khovanskii}, we get
\begin{align}\label{dim-Hk}
C_{G}\cdot\dim(H^0(M,L^k))=&k^n\int_{P_+}\pi(y)\,dy+\frac12k^{n-1}\left(\int_{\partial P_+}\pi\,d\sigma+2\int_{P_+}\langle\rho,\nabla \pi\rangle \,dy\right)\notag\\
&+o(k^{n-1}),~k\to+\infty,
\end{align}
where $d\sigma$ is the Lebesgue measure on each facet $F$ of $P_+$ normalized by the induced lattice $\mathfrak M|_F$ as in \cite[Section 3.2]{Do},
the constant $$C_G=\prod_{\alpha\in\Phi^G_+,\alpha\not\perp \mathfrak M_\mathbb R}{\langle\rho,2\alpha\rangle},$$ and
\begin{align}\label{pi}
\pi(y)=\prod_{\alpha\in\Phi^G_+,\alpha\not\perp \mathfrak M_\mathbb R}{\langle\alpha,y\rangle},~y\in P_+.
\end{align}
Here we consider $\mathfrak M$ as a sublattice of the character lattice $\mathfrak M^G$ of $G$. In the following we write $$\Phi^o:=\{\alpha\in\Phi^G|\alpha\not\perp \mathfrak M_\mathbb R\}~\text{and}~\Phi_+^o=\Phi^o\cap\Phi^G_+$$ for short. By definition, $P_+$ lies in the intersection of $\mathfrak M_\mathbb R$ with the (closed) dominate Weyl chamber of $\Phi^G_+$ in $ J\mathfrak t$.

Denote
$$P:=\displaystyle\bigcup_{w\in \mathcal W}w(P_+).$$
If $P$  is a convex polytope in $\mathfrak M_\mathbb R$,
we see that
$P_+=P\cap\mathcal V_+$.
We also have:
\begin{lem}\label{pi-extende-lem}
The function $\pi$ defined by \eqref{pi}  extends to a $\mathcal W$-invariant function on $P$:
\begin{align}\label{pi-extende}
\pi(y)=\prod_{\alpha\in\Phi_+^o}|{\langle\alpha,y\rangle}|,~y\in P.
\end{align}
\end{lem}
\begin{proof}
It suffices to show that \eqref{pi-extende} is $\mathcal W$-invariant since it coincides with \eqref{pi} in the fundamental chamber $\mathcal V_+$. By \cite[Lemma 4.1]{Knop14}, for each $w\in\mathcal W$, there is a $\sigma$ in the Weyl group $W$ of $(G,T^\mathbb C)$ so that $w=\sigma|_{\mathfrak M_\mathbb R}$. Recall that $\langle\cdot,\cdot\rangle$ is $W$-invariant. Hence the restriction of $\langle\cdot,\cdot\rangle$ on $\mathfrak M_\mathbb R$ is $\mathcal W$-invariant.

It then suffices to show that for each $w\in\mathcal W$, the corresponding $\sigma$ preserves $\Phi^o$. Since it always holds $\sigma(\Phi^G)=\Phi^G$, this reduces to prove that if $\alpha\in\Phi^G$ is orthogonal to $\mathfrak M_\mathbb R$, so is $\sigma(\alpha)$.

By the remark after \cite[Lemma 4.1]{Knop14}, $\sigma$ is either the reflection $s_\alpha$ with respect to some root $\alpha\in\Phi^G\cap\mathfrak M_\mathbb R$ or $\sigma=s_\alpha\circ s_\beta$ for orthogonal roots $\alpha,\beta\in\Phi^G$. Obviously $\sigma$ preserves $\Phi^o$ in the previous case. In the later case, as $\mathfrak M_\mathbb R$ is $\sigma$-invariant, we have
\begin{align}\label{MR-contain}
\mathfrak M_\mathbb R\subset\text{Span}_\mathbb R\{\alpha+\lambda_0\beta,W_\alpha\cap W_\beta\},
\end{align}
where $\lambda_0\not=0$ is a constant and $W_\alpha,W_\beta$ are the Weyl walls defined by $\alpha$ and $\beta$, respectively. On the other hand, for any $v\in\mathfrak M_\mathbb R^G$ decomposed as
$$v=v_o+v_\perp,~v_o\in\text{Span}_\mathbb R\{\alpha,\beta\}~\text{and}~v_\perp\in W_\alpha\cap W_\beta,$$
we have $$\sigma(v)=-v_o+v_\perp.$$
Combining with \eqref{MR-contain}, we see that if $\alpha\in\Phi^G$ is orthogonal to $\mathfrak M_\mathbb R$, then both $\alpha_o$ and $\alpha_\perp$ are orthogonal to $\mathfrak M_\mathbb R.$ Hence $\sigma(\alpha)\perp\mathfrak M_\mathbb R$ and we get the lemma.
\end{proof}

\begin{rem}
In general, it is not always true that $P$
 is a convex polytope in $\mathfrak M_\mathbb R$.
However there are still many interesting cases in which this is true. For example, the case of group compactifications. In the Appendix we give a sufficient condition so that this condition holds. We also refer readers to \cite[Corollary 5.14]{Del4} for characterization of horosymmetric varieties satisfying that $P$ is a convex polytope.
\end{rem}

\begin{rem}
We have $$P_+\subset\mathfrak M_\mathbb R\cap\{\pi\geq0\}\subset\mathcal V_+^*.$$
Here $\mathcal V_+^*$ is the dominate Weyl chamber of $\mathcal W$ in $\mathfrak M_\mathbb R$. By Lemma \ref{pi-extende-lem}, if $P_+$ extends to a $\mathcal W$-invariant convex polytope $P$, then it must hold $\pi=0$ on $\partial \mathcal V^*_+$.
\end{rem}

\begin{rem}
For a $\mathcal G$-compactification $M$, the moment polytope $P_+$ lies in the positive Weyl chamber $\mathfrak a_+(\cong\mathcal V_+)$ in $\mathfrak M^{\mathcal G}_\mathbb R$. It can be written as the intersection of a $W$-invariant convex polytope $P$ with $\mathfrak a_+$ (cf. \cite[Example 5.12]{Del3}). The decomposition \eqref{H0kL} can be reduced to
\begin{align*}
H^0(M,L^k)=\sum_{\lambda\in kP_+\cap\mathfrak M^\mathcal G}V_{\lambda}\otimes V^*_{\lambda}.
\end{align*}
Here $\lambda$ is considered as a dominate weight of $\mathcal G$. It is clear that in \eqref{dim-Hk}, the weight $\pi$ reduces to
$$\pi(y)=\prod_{\alpha\in\Phi_+^\mathcal G}{\langle\alpha,y\rangle^2},~y\in P_+,$$
and $$2\rho=\sum_{\alpha\in\Phi_+^\mathcal G}\alpha.$$
\end{rem}

\subsection{Equivariant test configurations and stability}
Recall that a  convex function is {\it piecewise linear function}
 if there exist finitely many linear functions
$\ell_1, \cdots, \ell_m$ such that $u(x)=\max\{\ell_1(x), \cdots, \ell_m(x)\}$. Moreover,
a piecewise linear function function $u$ is {\it rational} if all coefficients of $\ell_1,
\cdots, \ell_m$ are rational numbers.

The conception of K-stability of a polarized variety is usually stated in terms of test configurations. Let $(M,L)$ be a polarized $G$-spherical variety. A $G$-equivariant normal test configuration of $(M,L)$ consists of the following data:
\begin{itemize}
\item A normal $G\times \mathbb C^*$-variety $\tilde M$;
\item A $G\times \mathbb C^*$-linearized ample line bundle $\tilde L$ on $\tilde M$;
\item A $\mathbb C^*$-equivariant flat morphism $\tilde\pi:(\tilde M,\tilde L)\to\mathbb C$ so that the fibre $(\tilde M_i,\tilde L_i)$ over $i\in\mathbb C$ is isomorphic to $(M,L^{r_0})$ for some $r_0\in\mathbb N_+$.
\end{itemize}
One easily sees that $(\tilde M,\tilde L)$ is indeed a polarized $G\times \mathbb C^*$-spherical variety. If the total space $\tilde M$ is $(G\times \mathbb C^*)$-equivariantly isomorphic to $M\times \mathbb C$, we will call $(\tilde M,\tilde L)$ a \emph{product test configuration}.

As a generalization of \cite[Section 2.4]{AK} and \cite[Theorem 3.30]{Del3}, it is proved in \cite[Section 4]{Del-202009} that:
\begin{prop}\label{TC-classification}
Let $(M,L)$ be a polarized $G$-spherical variety. Then normal $G$-equivariant test configurations of $(M,L)$ are in one-one correspondence with rational convex piecewise linear functions (called the {associated functions}) on $P_+$ whose gradient lie in $\mathcal V_+$. Furthermore, a $G$-equivariant test configuration is a product test configuration  if and only if its associated function is an affine function with gradient in $\mathcal V_z$ (we call it a {central affine function}).
\end{prop}

Indeed, the total space of the normal $G$-equivariant test configuration associated to function $u$ has moment polytope that is precisely a polytope bounded by $P_+$ and the graph of  $R-u$, where $R$ is a sufficiently large constant. To define the K-stability, we recall the Futaki invariants of test configurations. It is defined by weights of the induced $\mathbb C^*$-action on Kodaira rings and we refer the readers to \cite{D3} and \cite{Del-202009} for the precise definition. Nevertheless, for a $G$-equivariant normal test configuration $(\tilde M, \tilde L)$, we may calculate the dimension of the weight space of the $\mathbb C^*$-action on $\tilde L$ as in \eqref{dim-Hk}. Therefore we get (cf. \cite[Section 3]{AK} and \cite[Section 5]{Del-202009}):
\begin{prop}\label{TC-futaki}
Let $(M,L)$ be a polarized $G$-spherical variety. Let $(\tilde M, \tilde L)$ be a normal $G$-equivariant test configurations of $(M,L)$ corresponding to the convex rational piecewise linear function $u$ on $P_+$. Then the Futaki invariant of $(\tilde M, \tilde L)$ is
\begin{align}\label{L(u)-org}
\text{Fut}(\tilde M, \tilde L)=\frac1{V}\left(\int_{\partial P_+}u\pi d\sigma-\int_{P_+}u\bar S\pi dy+2\int_{P_+}u\langle\rho,\nabla \pi\rangle \,dy\right)=:\mathcal L(u),
\end{align}
where $d\sigma$ is the lattice measure on each facet in $\partial P_+$, $V=\int_{P_+} \pi dy$ is the volume of $(M,L)$ and
$$\bar S=\frac1{V}\left(\int_{P_+}\pi dy+2\int_{P_+}\langle\rho,\nabla \pi\rangle \,dy\right)$$
is the mean value of the scalar curvature.
\end{prop}

Now we can state the definition of K-stability:
\begin{defi}\label{stability-def}
We say that a polarized $G$-spherical variety $(M, L)$ is (equivariantly) K-semistable if the Futaki invariant for any $G$-equivariant test configuration  is nonnegative, and is (equivariantly) K-stable if in addition the Futaki invariant vanishes precisely on product $G$-equivariant test configurations. An (equivariantly) K-semistable but not (equivariantly) K-semistable $(M,L)$ is said to be strictly (equivariantly) K-semistable. When $(M,L)$ is not (equivariantly) K-semistable, we say it is K-unstable.
\end{defi}
Note that on a K-semistable $(M,L)$,  the Futaki invariant always vanishes on product $G$-equivariant test configurations. By Propositions \ref{TC-classification} and \ref{TC-futaki}, this is equivalent to that the functional $\mathcal L(u)$ vanishes whenever $u$ is a central affine function.


In the case when $(M,L)$ is K-unstable, we may consider variants of Futaki invariants and the corresponding variants of K-stability. For example, by taking a variant of $\mathcal L$-functional \eqref{L(u)} below instead of \eqref{L(u)-org}, we can define the relative K-stability analogous to Definition \ref{stability-def}. This is closely related to the existence of extremal K\"ahler metrics (cf. \cite[Section 2]{Szekelyhidi-BLMS-2007}).

Assume that $P=\displaystyle\bigcup_{w\in\mathcal W}w(P_+)$ is a convex polytope and is given by
$$P=\bigcap_{A=1}^{d_0}\{y\in\mathfrak M_\mathbb R|\lambda_A-u_A(y)\geq0\},$$
so that each facet $$F_A\subset\{y\in\mathfrak M_\mathbb R|\lambda_A-u_A(y)=0\},A=1,...,d_0.$$
By $\mathcal W$-invariance, if $w(F_A)=F_{A'}$ for some $w\in\mathcal W$, we have
$\lambda_A=\lambda_{A'}$  and $u_A=w(u_{A'})$.
We may assume that origin $O\in P$ so that each $\lambda_A\geq0$. Indeed, by $\mathcal W$-invariance, $\lambda_A=0$ only if $\pi|_{F_A}=0$ and $O$ lies in $\mathcal V_z$. We can shift $P$ so that $O$ lies in the relative interior $\text{RelInt}(\mathcal V_z\cap P)$ of $\mathcal V_z\cap P$.

The mean value of scalar curvature $\overline S$ is determined by the positive $\lambda_A$'s as in \cite[eq. (1.3)]{LZZ}. In particular,
$\bar S>0$.
Let $X$ be the extremal vector field of $(M,L)$. It is a vector in $\mathcal V_z$ and its potential $\theta_X$ can be determined by \cite[Lemma 4.2]{LZ}. In particular, $\theta_X$ can be reduced to a central affine function (still denoted by $\theta_X$) on $\mathfrak M_\mathbb R$ (cf. \cite{LZ} or \cite{Del3}).

As in \cite[Proposition 3.1]{LZZ}, the relative Futaki invariant can be reduced to
\begin{align}\label{L(u)}
\mathcal L_X(u)=\sum_{A=1}^{d_+}\frac1{\lambda_A}\int_{F_A\cap P_+}u\langle y,\nu_A\rangle\pi \,d\sigma_0-\int_{P_+}u(\overline S+\theta_X)\pi \,dy+\int_{P_+}u\langle2\rho,\nabla \pi\rangle \,dy,
\end{align}
where $\{F_A\}_{A=1,...,d_+}$ are the outer facets of $P_+$, $d\sigma_0$ is the standard Lebesgue measure on $F_A$ and $\nu_A$ the outer unit normal vector. Recall that a facet $F_A$ is called an \emph{outer facet} if its relative interior intersects the interior of $\mathcal V_+$. The outer normal vector $u$ of a facet lies in $\mathcal V_+$. We highlight that when $\lambda_A\not=0$, the lattice measure $d\sigma$ in \eqref{L(u)-org} is
$$d\sigma=\frac1{\lambda_A}d\sigma_0.$$
Hence, it simply holds
$$\mathcal L_X(u)=\mathcal L(u)-\int_{P_+}u\theta_X\pi \,dy.$$
By the definition of the extremal vector field and the last part of Proposition \ref{TC-classification}, $X$ is uniquely determined by the condition $\mathcal L_X(\cdot)=0$ for any central affine functions. In particular, $X$ vanishes when $(M,L)$ is K-semistable.

\subsection{Dominated convex funcitons}
Let $\mathfrak a$ be a finitely dimensional real linear space. Let $\mathcal W$ be the Weyl group of some roots system $\Phi$ in $\mathfrak a$. Choose a system of dominate roots $\Phi_+\subset\Phi$ and denote by $\mathfrak a_+$ the closed positive Weyl chamber of $\Phi_+$. The notion of $\mathcal W$-dominate convex functions will play an important role in our later sections.
It can be defined on general $\mathcal W$-invariant convex domains.
\begin{defi}\label{dominate-convex-function-def}
Let $\Omega$ be a $\mathcal W$-invariant convex domain and $\Omega_+:=\Omega\cap{\mathfrak a_+}$. A convex function $\psi:\Omega_+\to\mathbb R$ is said to be ($\mathcal W$-)dominate, if it extends to a convex function on the whole $\Omega$ by the $\mathcal W$-action.
\end{defi}
\begin{rem}
For simplicity, for a dominate convex function $\psi$ on $\Omega_+$, when there is no confusions we do not distinguish $\psi$ and its extension on $\Omega$ via the $\mathcal W$-action.
\end{rem}

The following is a simple criterion for $\mathcal W$-dominate convex functions:
\begin{prop}\label{dominate-convex-function}
A convex function $\psi:\Omega_+\to\mathbb R$ is $\mathcal W$-dominate if and only if its normal mapping $\partial\psi(x)$ at any $x\in\Omega_+$ is contained in ${\mathfrak a_+}$.
\end{prop}

\begin{proof}
Let us first prove the ``if" part. We assume that $\psi$ is convex on $\Omega_+$ and $\partial\psi(x)$ at any $x\in\Omega_+$ is contained in $\Omega$.

By convexity,
$$\psi(x)=\sup\{l(x)|l~\text{is a support function of}~\psi~\text{on}~\Omega_+\}.$$
On the other hand, by our assumption, for each possible $l$ above, $\nabla l\in\overline{\mathfrak a_+}$. Hence $l$ extends to a $\mathcal W$-invariant convex function
$$\tilde l(x):=\max_{w\in \mathcal W}\{l(w\cdot x)\}$$
on $\Omega$ and $\psi$ extends to a $\mathcal W$-invariant convex function
$$\tilde\psi(x)=\sup\{\tilde l(x)|l~\text{is a support function of}~\psi~\text{on}~\Omega_+\}$$
on the whole $\Omega$.

Now we prove the ``only if" part. For any $y_0\in\Omega_+$ and $\alpha\in \Phi_+$, we can find a point $y_0'$ in the Weyl wall $W_\alpha$ so that $y_0=y_0'+t_0\alpha$ for some $t_0\geq0$. By $\mathcal W$-invariance and convexity, the function
$$f(t)=\psi(y_0'+t\alpha),~-t_0\leq t\leq t_0$$
is a convex even function of $t$. Hence $t_0=0$ is a minima. By  convexity, we get
$$\langle\alpha,\nabla\psi(y_0)\rangle=f'(t_0)\geq0.$$
Hence we have proved the proposition.
\end{proof}


\vskip 20pt

\section{Minimizers of variational problems with a convexity constraint}\label{mini}

As can be seen in the previous section, the K-stability of polarized spherical varieties can be reduced to a variational problem for the functional $\mathcal L_X(u)$ for convex functions. There are many variational problems with a convex constraint for functionals of the form
\beq\label{F-fun}
\mathcal F(u)=\int_\Omega f(x, u, \nabla u)\,d\mu
\eeq
arising from physics,  elasticity, economics, and so on.
 In general, due to the convexity constraint, there is no explicit Euler-Lagrange equation for this kind of problems. An untraceable representation of the Euler-Lagrange equation was obtained by \cite{Car, Lions}.
In this section, we show that under certain conditions, the homogenous Monge-Amp\`ere equation (HMA) is
the Euler-Lagrange equation. This idea appeared first in \cite{Wang-Zhou}, where we deal with the functional $\mathcal L_X(\cdot)$ for toric manifolds. Here we extend it to a more general setting that might also be useful in other places.

Let $\Omega\subset\mathbb R^n$ be a bounded convex domain with piecewise smooth boundary. Denote by $\Omega^*$ the union of $\Omega$ and the relative interior of each smooth piece of $\partial\Omega$. 
Suppose $d\sigma=\sigma(x)\, d\sigma_0$ and $d\mu=\mu(x)\,dx$ are two signed measures on $\p\Omega$ and $\Omega$,
respectively. Here $d\sigma_0$ is the standard Lebegue measure on $\p\Omega$.
Consider a functional
$$\mathcal F(u)=\int_{\partial\Omega}\mathcal A(u)\,d\sigma+\int_\Omega\mathcal B(u)\,d\mu$$
where $\mathcal A,\mathcal B\in C^0(\mathbb R)$, and $u$ lies in a certain class of convex functions. Assume that:
\begin{itemize}
\item[(P1)] $\sigma(x)\geq0$ is a piecewise smooth function defined in a neighborhood of $\partial\Omega$ and is smooth on the relative interior of each smooth piece of $\partial\Omega$;
\item[(P2)] $\mu(x)$ is a smooth function on $\overline\Omega$ and $\Omega_-:=\{x\ |\mu(x)<0\}\not=\emptyset$ is the union of finitely many disjoint domains with Lipschitz boundary;
\item[(P3)] The zero set of $\mu(x)$ has Lebesgue measure $0$;
\item[(P4)] $\mathcal A$ is increasing and $\mathcal B$ is strictly increasing;
\item[(P5)] $\mathcal F(u)>-\infty$ whenever $\int_{\partial\Omega}\mathcal A(u)\,d\sigma<+\infty$.
\end{itemize}
Set a class of convex functions
\begin{align*}
\mathcal C(\Omega):=\{u:\Omega^*\to\mathbb R \,|\ u~\text{is convex},\int_{\partial\Omega}\mathcal A(u)\,d\sigma<+\infty\}.
\end{align*}
Then $\mathcal F(\cdot)$ is well-defined on $\mathcal C(\Omega)$. In order to derive the equation for the minimizers, we first recall the following well-known characterization of solutions to the homogenous Monge-Amp\`ere equation.

\begin{lem}{\cite[Theorem 2.1]{TrW}}\label{hma1}
Let $\Omega$ be a Lipschitz domain in $\mathbb R^n$, and $u_0$ be a continuous function on $\overline\Omega$.
Then
\beq 
u(x)=\sup\{\ell(x) | \ell\ \text{is a linear function in $\overline\Omega$
         with $\ell\le u_0$ on $\partial\Omega$}\}
 \eeq
is the unique convex solution (generalized solution in the sense of
Aleksandrov) to the Monge-Amp\`ere equation
 \beq\label{HMA-eq-int-Omega}
 \det (\partial^2 u)=0
  \eeq
in $\Omega$, subject to the Dirichlet boundary condition $u=u_0$ on
$\partial\Omega$.
\end{lem}

\begin{prop}\label{HMA-prop}
Suppose that $\mathcal F(\cdot)$ satisfies the assumptions (P1)-(P5) and $u_0\in\mathcal C(\Omega)$ is a minimizer of $\mathcal F(\cdot)$. Then  $u_0$ must satisfy the homogenous Monge-Amp\`ere equation
\eqref{HMA-eq-int-Omega}
 on $\Omega$ as a generalized solution in the sense of Alexandrov.
\end{prop}

\begin{proof}
Put $\Omega_+:=\Omega\setminus\Omega_-$. Then both $\Omega_+$ and $\Omega_-$ have Lipschitz boundary. Consider
$$u_+(x):=\sup\{\ell(x)| \ell~\text{is affine and}~\ell|_{\Omega_+\cup\partial\Omega}\leq u_0|_{\Omega_+\cup\partial\Omega}\}.$$
Clearly $u_+\in\mathcal C(\Omega)$. By Lemma \ref{hma1}, $u_+$ satisfies the Homogenous Monge-Amp\`ere equation \eqref{HMA-eq-int-Omega}
on $\Omega_-$ and
\begin{align*}
u_+&\geq u_0~\text{\ on $\Omega_-$},\\
u_+&=u_0~\text{\ on $\Omega_+\cup\partial\Omega$}.
\end{align*}
We claim that $u_0=u_+$ on $\Omega^*$. Otherwise, there is a point $x_0\in\Omega_-$ so that $u_0<u_+$ in a neighborhood $U$ of $x_0$. By (P2), $U$ can be chosen sufficiently small such that $\mu(x)$ has strictly negative upper bound on $U$. Combining with (P4),
\begin{align*}
\mathcal F(u_+)=&\mathcal F(u_0)-\int_{\Omega_-\setminus U}(\mathcal B(u_0)-\mathcal B(u_+))\mu(x)\,dx\\
&-\int_{U}(\mathcal B(u_0)-\mathcal B(u_+))\mu(x)\,dx<0.
\end{align*}
A contradiction to the fact that $u_0$ is a minimizer.

Next, let
\begin{align}\label{u-}
u_-(x):=\sup\{\ell(x)|\ell~\text{is a support function of}~u_0|_{\Omega_-}\}.
\end{align}
Then
\begin{align*}
u_-&\leq u_0~\text{on $\Omega_+$},\\
u_-&=u_0~\text{on $\Omega_-$}.
\end{align*}
By (P4) and the second inequality $u_-\in\mathcal C(\Omega)$. We claim that $u_0=u_-$ on $\Omega^*$. Otherwise, there is a point $x_0\in\Omega_+$ so that $u_0>u_+$ in a neighborhood $U$ of $x_0$. By (P3), $U$ can be chosen sufficiently small such that $\mu(x)$ has strictly positive lower bound on $U$. Combining with (P4),
\begin{align*}
\mathcal F(u_-)=&\mathcal F(u_0)-\int_{\Omega_+}(\mathcal B(u_0)-\mathcal B(u_-))\mu(x)\,dx-\int_{\overline{\Omega_+}\cap\partial \Omega}(\mathcal A(u_0)-\mathcal A(u_-))\sigma\,d\sigma\\
\leq &-\int_{U}(\mathcal B(u_0)-\mathcal B(u_+))\mu(x)\,dx<0.
\end{align*}
Again this contradicts to the fact that $u_0$ is a minimizer.

It remains to show that $u_0$ satisfies \eqref{HMA-eq-int-Omega}.
Note that the function $u_-$
in \eqref{u-} can be approximated by
\begin{equation}\label{unege}
 u_-^\epsilon(x)=\sup\{\ell(x)| \ell\
 \text{is a supporting function of $u$ at some point $x\in
 \Omega_-^\epsilon$}\},
 \end{equation}
where $\Omega_-^\epsilon=\{x\in \Omega| \Phi(x)\leq-\epsilon\}\subset \Omega_-$. By \eqref{u-}, a
supporting plane of $u_-^\epsilon$ at some point in $\Omega$ must also be a
supporting plane of $u_+$ at some point in $\Omega_-^\epsilon$. But since $u_+$
is a generalized solution of \eqref{HMA-eq-int-Omega} in $\Omega_-$, we have the Monge-Amp\`ere measure
$\text{MA}[u_+](\Omega^\epsilon_-)=0$ by definition. Hence $\text{MA}[u_-^\epsilon](\Omega)=0$ and so
$u_-^\epsilon$ is a generalized solution to \eqref{HMA-eq-int-Omega} in $\Omega$. By the weak convergence of the Monge-Amp\`ere measure, $u_0=\displaystyle\lim_{\epsilon\to 0}u_-^\epsilon$ is a generalized solution of \eqref{HMA-eq-int-Omega} in $\Omega$.
\end{proof}

In the later sections,
we will use some properties related to the real Monge-Amp\`ere equation,
which can be found in \cite{Gutierrez, TrW, Wang-Zhou}.
 Let $\Omega$ be a bounded convex domain in
$\mathbb R^n$, $n\ge 2$. A boundary point $z\in\partial\Omega$ is an {\it extreme point} of
$\Omega$ if there is a hyperplane $H$ such that $\{z\}=H\cap\partial\Omega$, namely
$z$ is the unique point in $H\cap\partial\Omega$. It is known that any
interior point of $\Omega$ can be expressed as a linear combination of
extreme points of $\Omega$. If $\Omega$ is a convex polytope, a boundary
point $z\in\partial\Omega$ is an extreme point if and only if it is a vertex
of $\Omega$.

Let $u$ be a convex function on $\Omega$. For any
interior point $z\in\Omega$, let $L_z=\{x\in\mathbb R^n |x_{n+1}=\phi(x)\}$ be
a supporting plane of $u$ at $z$. By convexity, the set $\mathcal
T:=\{x\in\Omega:\ u(x)=\phi(x)\}$ is convex.

\begin{lem}{\cite[Lemma 4.1]{Wang-Zhou}}\label{extreme}
Let $u$ be a generalized solution to \eqref{HMA-eq-int-Omega}.
Then an extreme point of $\mathcal T$ must be a
boundary point of $\Omega$.
\end{lem}

The above results is essentially used in \cite{Wang-Zhou} to discuss the K-stability of toric varieties.
Now we use it for general spherical varieties.
Let $\Omega=P$ be the $\mathcal W$-invariant polytope in the last section and
we investigate  the functional $\mathcal L_X(\cdot)$ given by \eqref{L(u)} on the following set of convex functions.
$$\mathcal C_{1,\mathcal W}=\{u: P^*\to\mathbb R|u~\text{is $\mathcal W$-dominate convex, continuous in $P^*$ and}~\int_{\partial P_+}u\pi \,d\sigma<+\infty\}.$$
Here $P^*$ is the union of the interior of $P$ and the interior of facets. The following Lemma shows that $\mathcal L_X(\cdot)$ satisfies the condition (P5):
\begin{lem}
$\mathcal L_X(\cdot)$ is well-defined on $\mathcal C_{1,\mathcal W}$.
\end{lem}
\begin{proof}
It suffices to check that $\mathcal L_X(u)>-\infty$ for any $u\in\mathcal C_{1,\mathcal W}$. Note that $O\in\text{RelInt}(P_+\cap\mathcal V_z)$ is a fixed point of the $\mathcal W$-action. By convexity and $\mathcal W$-invariance of $u$, up to adding a central affine function $\phi$, one can assume that
\begin{align}\label{normalize-1}
u\geq u(O)=0.
\end{align}
Since $\mathcal L_X(\phi)$ is finite for any affine $\phi$, it remains to show that $\mathcal L_X(u)>-\infty$ for any $u\in\mathcal C_{1,\mathcal W}$ satisfies \eqref{normalize-1}.

Recall \eqref{L(u)}. For simplicity, set
\begin{align}\label{theta}
\Theta(y):=\bar S+\theta_X-2\sum_{\alpha\in\Phi_+^o}\frac{\langle\alpha,\rho\rangle}{\langle\alpha,y\rangle},~y\in P_+.
\end{align}
Then
\begin{align*}
\mathcal L_X(u)=\sum_{A=1}^{d_+}\frac1{\lambda_A}\int_{F_A\cap P_+}u\langle y,\nu_A\rangle\pi \,d\sigma_0-\int_{P_+}u\Theta\pi \,dy,~\text{for}~u\in\mathcal C_{1,\mathcal W}.
\end{align*}
Since $\langle\alpha,\rho\rangle>0$ for any $\alpha\in\Phi_+^o$ and $\langle\alpha,y\rangle>0$ for any $y\in P_+$, there is a compact set $K\subset P_+$ that contains the union of all facets of $P_+$ where $\pi(y)=0$ so that
$$\Theta(y)<0,~y\in K.$$
Hence, whenever $u\geq0$,
$$\sum_{A=1}^{d_+}\frac1{\lambda_A}\int_{F_A\cap K} u\langle y,\nu_A\rangle\pi \,d\sigma_0-\int_{K}u\Theta\pi \,dy\geq0.$$

On the other hand, it is easy to see that $\Theta$ is bounded outside $K$. Hence by the normalization condition \eqref{normalize-1},
$$\sum_{A=1}^{d_+}\frac1{\lambda_A}\int_{F_A\setminus K} u\langle y,\nu_A\rangle\pi \,d\sigma_0-\int_{P_+\setminus K}u\Theta\pi \,dy\geq-C_{0}\int_{\partial P_+}u\pi d\sigma_0>-\infty,$$
for some uniform constant $C_{0}>0$. Hence we get the lemma.
\end{proof}

In the following we prove Theorem \ref{Main-thm1} (1).

\begin{prop}\label{HMA-eq-lem}
Suppose that $\mathcal L_X(\cdot)\geq0$ for any function in $\mathcal C_{1,\mathcal W}$ and there is $u_0\in\mathcal C_{1,\mathcal W}$ so that $\mathcal L_X(u_0)=0$. Then  $u_0$ must be a generalized solution to the homogenous Monge-Amp\`ere equation on the whole polytope $P$ in the sense of Alexandrov.
\end{prop}

\begin{proof}
Recall \eqref{theta}. By Lemma \ref{pi-extende-lem}, $\Theta$ can be extended to a $\mathcal W$-invariant function
$$\Theta(y)=\bar S+\theta_X-2\sum_{\alpha\in\Phi_+^o}\frac{\langle\alpha,\rho\rangle}{|\langle\alpha,y\rangle|},~y\in P.$$
By $\mathcal W$-invariance, up to dividing a constant $\#\mathcal W$, the functional $\mathcal L_X(\cdot)$ can be rewritten as
\begin{align}\label{L-Theta}
\mathcal L_X(u)=\sum_{A=1}^{d_0}\frac1{\lambda_A}\int_{F_A}u\langle y,\nu_A\rangle\pi \,d\sigma_0-\int_{P}u\Theta\pi \,dy,~\forall u\in\mathcal C_{1,\mathcal W}.
\end{align}

From the condition that $\mathcal L_X(1)=0$, we get
$$\int_{P}\Theta\pi \,dy=\sum_{A=1}^{d_0}\frac1{\lambda_A}\int_{F_A}\langle y,\nu_A\rangle\pi \,d\sigma_0>0.$$
Hence $\mathfrak A:=\{y\ |\Theta(y)\geq0\}$ is a $\mathcal W$-invariant non-empty set. Note that by \cite[Lemma 4.1]{Knop14}, when $y\in P_+$ goes to a facet of $\mathcal V_+$ (i.e. a Weyl wall with respect to some spherical root), there is at least one $\alpha\in \Phi_+^o$ so that $\langle\alpha,y\rangle\to 0^+$. Consequently, $\Theta\to-\infty$. Thus each component of $\mathfrak A$ lies in the interior of exactly one Weyl chamber of $\mathcal W$. By taking $\mu=-\Theta(y)\pi(y)$, $\mathcal L_X(\cdot)$ satisfies the conditions (P2)-(P3) above. It is also easy to see that each component is convex. Taking in addition $\mathcal A(u)=\mathcal B(u)=u$ in Proposition \ref{HMA-prop}, $u_0$ satisfies the homogenous Monge-Amp\`ere equation  on $P$. Hence we get the Proposition.
\end{proof}

\vskip 20pt


\section{A criterion of strict K-semistability when rank$(G/H)=2$}\label{r2}

Throughout this section, we assume that rank$(G/H)=2$.  We further assume that $P_+=P\cap\mathcal V_+$ for a $\mathcal W$-invariant convex polytope $P\subset\mathfrak M_\mathbb R\simeq\mathbb R^2$.
Following \cite{Do},
we call a function $u$ is {\it simple piecewise linear} if there is a linear function
$\ell$ such that $u=\max\{0,\ell\}$.
If $u$ is simple piecewise linear, the set  $\mathfrak I_{u}=P\cap \{\ell=0\}$
is called the {\it crease} of $u$. We establish a simpler criterion of strict relative K-semistability, which corresponds to Theorem \ref{Main-thm2} (2).

\begin{prop}\label{ker-thm}
If $\mathcal L_X(\cdot)\geq0$ for any function in $\mathcal C_{1,\mathcal W}$ and there is a non-central affine function $u_0\in\mathcal C_{1,\mathcal W}$ so that $\mathcal L_X(u_0)=0$, then there exists a $\mathcal W$-dominate simple piecewise linear function $\bar u\in\mathcal C_{1,\mathcal W}$, which is not a central affine function, so that $\mathcal L_X(\bar u)=0$.
\end{prop}

We will need the following elementary lemma:
\begin{lem}\label{ratio-lem}
For any $\mathcal W$-dominate simple piecewise linear function $\phi$, denote by
$$\mathfrak F_+(\phi):=\partial(\text{supp}(\phi))\cap\partial P\cap\partial P_+, \ \mathfrak F_0(\phi):=\partial(\text{supp}(\phi))\cap\partial \mathcal V_+$$ the
parts lying on the outer facets of $P_+$ and
 on the Weyl walls, respectively. Then there is a constant $c>0$ which depends only on $P$ such that the $1$-dimensional Lebesgue measure
\begin{align}\label{area-compare}
|\mathfrak F_+(\phi)|_1\geq c|\mathfrak F_0(\phi)|_1,~\forall~\text{$\mathcal W$-dominate simple piecewise linear}~\phi.
\end{align}
\end{lem}

\begin{proof}
Since $\dim(\mathfrak M_\mathbb R)=2$, there are at most two simple roots in $\Phi_+$. We will only deal with the case that there are precisely two simple roots since the remaining cases are much simpler. Denote by $\alpha_1,\alpha_2$ the two simple roots and $\varpi_1,\varpi_2$ the corresponding fundamental weights so that $$\langle\alpha_i,\varpi_j\rangle=\frac12|\alpha_i|^2\delta_{ij},~1\leq i,j\leq2.$$
Recall that $\mathcal V_+$ is the dominate Weyl chamber of the positive spherical roots $\Phi_+$, $$\mathcal V_+=\text{Span}_{\mathbb R_{\geq0}}\{\varpi_1,\varpi_2\}.$$ Hence $\partial\mathcal V_+$ consists of two rays $\mathbb R_{\geq0}\varpi_i,i=1,2$. Indeed, $\mathbb R_{\geq0}\varpi_i\subset W_{\alpha_j}$ for $i\not=j$.

A dominate simple piecewise linear function can be written as
\begin{align}\label{phi-lem}
\phi(y)=\max\{\Lambda(y)-\lambda,0\},
\end{align}
where the gradient $\Lambda\in{\mathcal V_+}$. We may write
\begin{align}\label{Lambda-decomp}
\Lambda=c_1\varpi_1+c_2\varpi_2,~c_1,c_2\geq0.
\end{align}
Its crease
$$\mathfrak I_\phi=\{y\in \mathcal V_+|\Lambda(y)=\lambda\}$$
is a line segment with normal vector $\Lambda$. If we order $\varpi_1,\varpi_2$ counterclockwise, then $\text{supp}(\phi)$ is just the upper right hand side of $\mathfrak I_\phi$ and $\mathfrak F_+$ ($\mathfrak F_0$, resp.) is the upper right hand side part of $\partial P_+\cap\partial P$  (Weyl walls, resp.) cut out by $\mathfrak I_\phi$.

Note that $\langle\varpi_1,\varpi_2\rangle\geq0$. By \eqref{Lambda-decomp}, $\langle\Lambda,\varpi_i\rangle\geq0$ for $i=1,2$. When there are constants $\epsilon_i>0$, $i=1,2$, so that
$$\langle\frac{\Lambda}{|\Lambda|},\varpi_i\rangle\geq\epsilon_i>0,$$
there is a constant $c_i=c_i(P,\epsilon_i)$ so that
$$|\mathfrak F_+|_2\geq c_i|\mathfrak F_0\cap W_{\alpha_j}|,~i\not=j.$$
By \eqref{Lambda-decomp}, when $\varpi_1\not\perp\varpi_2$, such $\epsilon_1$ and $\epsilon_2$ always exist, and we get the lemma.


It remains to deal with the case when $\varpi_1\perp\varpi_2$ and $\langle\frac{\Lambda}{|\Lambda|},\varpi_1\rangle\leq \epsilon\ll1$. This can only happen when $0<{c_1}\ll{c_2}$ in \eqref{Lambda-decomp}. Since $c_2>0$, $\langle\frac{\Lambda}{|\Lambda|},\varpi_2\rangle$ is almost $|\varpi_2|>0$. The set
$$\text{supp}(\phi)\cap W_{ \alpha_2}=\{t\varpi_1|t\geq\frac{\lambda}{c_1|\varpi_1|^2}\}.$$
However, since $c_1\ll c_2=O(1)$,
$$|\text{supp}(\phi)\cap W_{ \alpha_2}|_1=o(|P\cap W_{\alpha_2}|_1),~\text{as}~\frac{c_1}{c_2}\to0^+.$$
Hence, in this case,
$$|\mathfrak F_0|_1=O(|\mathfrak  F_0\cap W_{\alpha_1}|_1),~\text{as}~\frac{c_1}{c_2}\to0^+.$$
But $\langle\frac{\Lambda}{|\Lambda|},\varpi_2\rangle>\frac12|\varpi_2|>0$ if $\langle\frac{\Lambda}{|\Lambda|},\varpi_1\rangle$ is sufficiently small. By the arguments of the previous case we can choose $\epsilon_2>0$ such that
$$|\mathfrak F_+(\phi)|_1\geq\epsilon_2 |\text{supp}(\phi)\cap W_{\alpha_1}|_1.$$
The lemma then follows from the above two relations.
\end{proof}

\begin{lem}\label{small-support-lem}
There is a constant $\delta=\delta(P)>0$ which depends only on $P$ such that for a $\mathcal W$-dominate simple piecewise linear function $u$ on $P_+$ satisfies $|\text{supp}(u)|_2<\delta(P)$, it holds $\mathcal L_X(u)>0$.
Here by $|\cdot|_2$ we denote the $2$-dimensional Lebesgue measure of a set in $\mathfrak M_\mathbb R$.
\end{lem}

\begin{proof}
Denote the two classes of facets of $P_+$:
\begin{align*}
\mathcal F_1&=\{F\ | \ F~\text{is a facet of $P_+$ so that}~\pi|_{F}=0\},\\
\mathcal F_2&=\{F\ | \ F~\text{is a facet of $P_+$ so that}~\pi|_{F}\not=0\}.
\end{align*}
By \eqref{pi}, the zero set of $\pi$ is a union of lines passing trough $O\in \mathfrak M_\mathbb R$. Also since $P_+$ only contains dominate weights of $G$, when rank$(G)=2$,  $\mathcal F_1$ can contain at most two facets. Otherwise a Weyl wall defined by some $\alpha\in\Phi_+^G$ will pass through the interior of $P_+$ and there is a non-dominate weight lies in $P_+$.

Let
$$u(y)=\max\{\Lambda(y)-c,0\},~y\in P_+$$ be
a $\mathcal W$-dominate simple piecewise linear function with $\Lambda\in \mathcal V_+$.
Then for $\delta>0$ sufficiently small, if $|\text{supp}(u)|_2<\delta$,
\begin{align}\label{measure-compare}
|\text{supp}(u)|_2=o(|\text{supp}(u)\cap\partial P_+|_1),~|\text{supp}(u)\cap\partial P_+|_1\to0.
\end{align}

We first deal with the most complicated case when $\dim(\mathcal V_z)=2$. That is, $\mathcal W$ is trivial. In this case the valuation cone $\mathcal V_+=\mathfrak N_\mathbb Q$, the polytope $P_+=P$ and $\mathfrak F_0(\cdot)$ is always empty. We have three cases:

\vskip 10pt
\emph{Case-1.} $\mathcal F_1=\emptyset$. Note that $\rho$ is a dominate weight. Hence
\begin{align}\label{2rho-nabla-pi}
\langle2\rho,\nabla\pi\rangle=\sum_{\alpha\in\Phi_+^o}\frac{2\langle\rho,\alpha\rangle}{\langle\alpha,y\rangle}\pi(y)\geq0.
\end{align}
Since $u\geq0$, by \eqref{L(u)},
\begin{align}\label{L(u)-lower-bound}
\mathcal L_X(u)\geq\sum_{A=1}^{d_+}\frac1{\lambda_A}\int_{F_A\cap P_+}u\langle y,\nu_A\rangle\pi \,d\sigma_0-\int_{P_+}u(\overline S+\theta_X)\pi \,dy.
\end{align}
Combining with \eqref{measure-compare} we get the conclusion.

\vskip 10pt

\emph{Case-2.} $\mathcal F_1$ contains precisely one facet $F_0$ orthogonal to $\alpha_0\in\Phi_+^o$. For simplicity we assume there is only one root in $\Phi_+^o$ that orthogonal to $F_0$. When there are at least two roots, they  have common zero set $F_0$ in $P_+$ and the lemma can be proved in a same way. Since $P_+$ is the convex hull of dominate $G$-weights, $\langle \alpha_0,y\rangle\geq0$ for any $y\in P_+$. Denote by $\alpha_0'$ the projection of $\alpha_0$ in $\mathfrak M_\mathbb R$. It is a non-zero vector since $\alpha_0\in\Phi_+^o$. Then $\alpha_0'$ is an inner normal vector of $F_0$ (see Figure 1).

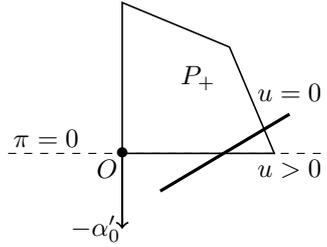
\begin{figure}[h]
\begin{tikzpicture}
\draw[semithick] (0,0) -- (0,2) -- (1.41,1.41) -- (2,0)--(0,0);
\draw[dashed] (-1.5,0) -- (2.75,0);
\draw (-1,0.2) node{$\pi=0$};
\draw[very thick]  (0.5,-0.5)  -- (2.2,0.4+0.12);
\draw (2.2,0.8) node{$u=0$};
\draw (2.2,-0.2) node{$u>0$};
\draw[->, thick] (0,0) -- (0,-1);
\draw (-0.2,-0.2) node {$O$};
\draw (-0.35,-1) node {$-\alpha_0'$};
\draw (1,1) node {$P_+$};
\draw (0,0) node{$\bullet$};
\end{tikzpicture}
\caption{\emph{The case when $1-\tau\geq\epsilon_0>0$.}}
\end{figure}

Denote $$\tau=-\frac{\langle\Lambda,\alpha_0'\rangle}{|\Lambda||\alpha_0'|}\in[-1,1].$$
First, we consider the case when $1-\tau\geq\epsilon>0$ for some sufficiently small $\epsilon>0$,
which means the angle between $\Lambda$ and $\alpha_0'$ is not too small (see Figure 1).  Then there are constants $c_0=c_0(\epsilon, P)$, $\delta=\delta(\epsilon, P)>0$ so that
$$|\mathfrak F_+(u)\setminus F_0|_1\geq c_0|\text{supp}(u)\cap F_0|_1,~\text{if}~|\text{supp}(u)|_2<\delta\ll1.$$
Again by \eqref{L(u)-lower-bound}, we can choose $\delta$ sufficiently small such that $\mathcal L_X(u)>0$. It suffices to consider the case when $1-\tau\leq \epsilon$. That is, $\Lambda$ is almost parallel to $-\alpha_0'$. In fact, when $\epsilon$ is small,  there exists $c_1=c_1(P,\epsilon)>0$ such that (see Figure 2):
\begin{align*}
\text{supp}(u)\subset\mathcal S_{\alpha_0}(c_1\delta):=P_+\cap\{y|0\leq\langle\alpha_0,y\rangle\leq c_1\delta\}.
\end{align*}

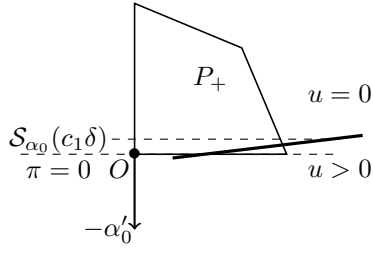
\begin{figure}[h]
\begin{tikzpicture}
\draw[semithick] (0,0) -- (0,2) -- (1.41,1.41) -- (2,0)--(0,0);
\draw[dashed] (-1.5,0) -- (2.7,0);
\draw[dashed] (-0.3,0.2) -- (2.7,0.2);
\draw (-1,-0.2) node{$\pi=0$};
\draw (-1,0.2) node{$\mathcal S_{\alpha_0}(c_1\delta)$};
\draw[very thick]  (0.5,-0.05)  -- (3,0.25);
\draw (2.7,0.8) node{$u=0$};
\draw (2.7,-0.2) node{$u>0$};
\draw[->, thick] (0,0) -- (0,-1);
\draw (-0.2,-0.2) node {$O$};
\draw (-0.35,-1) node {$-\alpha_0'$};
\draw (1,1) node {$P_+$};
\draw (0,0) node{$\bullet$};
\end{tikzpicture}
\caption{\emph{The case when $1-\tau\ll1$.}}
\end{figure}

Note that in the trip $\mathcal S_{\alpha_0}(c_1\delta)$, we have
\begin{align}\label{compare-order}
(\overline S+\theta_X)\pi=o(\langle\rho,\nabla \pi\rangle),~\text{as}~\delta\to0^+,
\end{align}
since the left-hand side contains at least one more $\langle\alpha_0,y\rangle$-factor than in $\langle\rho,\nabla \pi\rangle$. By \eqref{L(u)},
$$\mathcal L_X(u)\geq-\int_{P_+}u(\overline S+\theta_X)\pi \,dy+\int_{P_+}u\langle2\rho,\nabla \pi\rangle \,dy,~\forall u\in\mathcal C_{1,\mathcal W}.$$
Combining with \eqref{2rho-nabla-pi} and \eqref{compare-order}, we see that $\mathcal L_X(u)>0$ if $|\text{supp}(u)|_2$ is small enough. In conclusion, the proof is completed in this case.

\vskip 10pt

\emph{Case-3.} $\mathcal F_1$ contains two facets $F_1, F_2$ orthogonal to the roots $\alpha_1$ and $\alpha_2$ in $\Phi_+^o$, respectively. We may assume $F_1$ and $F_2$ has a common point (it has to be $O$). Otherwise, since $|{\rm supp} (u)|$ is very small, we can assume that at $u$ vanishes on at least one of $\{F_1,F_2\}$. Hence we can deal the case as in \emph{Case-2} near the other facet.

Denote
$$\tau_i=-\frac{\langle\Lambda,\alpha_i'\rangle}{|\Lambda||\alpha_i'|}\in[-1,1],~i=1,2,$$
and the cone
$$\mathfrak C=\text{Span}_{\mathbb R_{\geq0}}\{-\alpha_1,-\alpha_2\}\cap\mathcal V_+.$$

\emph{Case-3.1.} If $1-\tau_i\geq\epsilon>0$ for both $i=1,2$  for some $\epsilon$ and $\Lambda\not\in\mathfrak C$, then there are constants $c_0=c_0(\epsilon, P)$, $\delta=\delta(\epsilon, P)>0$ so that
$$|\mathfrak F_+(u)\setminus (F_1\cup F_2)|_1\geq c_0|\text{supp}(u)\cap (F_1\cup F_2)|_1,~\text{if}~|\text{supp}(u)|_2<\delta\ll1.$$
As in the first subcase of \emph{Case-2}, by \eqref{L(u)-lower-bound} we see that $\mathcal L(u)>0$ if $|\text{supp}(u)|_2\ll1$ (see Figure 3).

\begin{figure}[h]
\begin{tikzpicture}
\draw[semithick] (0,0)--(2,2) --(3,1)-- (3,0)--(0,0);
\draw[dashed] (-1.5,0) -- (3.5,0);
\draw[dashed] (-1.2,-1.2) -- (2.2,2.2);
\draw (-1,-0.2) node{$\pi=0$};
\draw (1.2,1.8) node{$\pi=0$};
\draw[very thick]  (-0.20,-0.7)  -- (3.5,0.5);
\draw (3.5,0.8) node{$u=0$};
\draw (3.5,0.2) node{$u>0$};
\draw[->, semithick] (0,0) -- (0,-1);
\draw[->, semithick] (0,0) -- (-1,1);
\draw (-0.2,-0.2) node {$O$};
\draw (-0.35,-1) node {$-\alpha_1'$};
\draw (-1,1.2) node {$-\alpha_2'$};
\draw (2,0.75) node {$P_+$};
\draw (0,0) node{$\bullet$};
\end{tikzpicture}
\caption{\emph{Case-3.1: The case that $1-\tau_i\geq\epsilon>0$ for both $i=1,2$  for some $\epsilon$ and $\Lambda\not\in\mathfrak C$.}}
\end{figure}
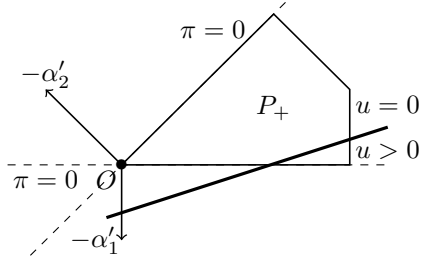

\emph{Case-3.2.} It remains to consider the case when there is some $i$ such that $1-\tau_i\leq \epsilon\ll1$ (see Figure 4) or $\Lambda\in\mathfrak C$ (see Figure 5).  In either case, there exists $c_1=c_1(P,\epsilon)>0$ such that
 as $|\text{supp}(u)|_2<\delta$ for sufficiently small $\delta$, it holds
$$\mathfrak F_+(u)\subset \mathcal S_{\alpha_1}(c_1\delta)\cup\mathcal S_{\alpha_2}(c_1\delta)$$
Here the two strips
\begin{align*}
\mathcal S_{\alpha_i}(c_1\delta)=P_+\cap\{y|0\leq\langle\alpha_i,y\rangle\leq c_1\delta\},~i=1,2.
\end{align*}
The conclusion can be obtained again by \eqref{compare-order}.

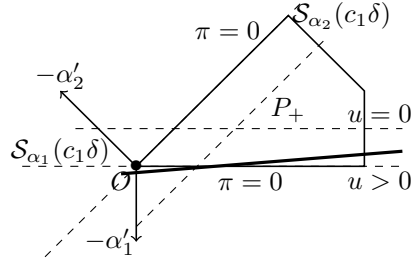
\begin{figure}[h]
\begin{tikzpicture}
\draw[semithick] (0,0)--(2,2) --(3,1)-- (3,0)--(0,0);
\draw[dashed] (-1.5,0) -- (3.5,0);
\draw[dashed] (-0.8,0.5) -- (3.5,0.5);
\draw[dashed] (-1.2,-1.2) -- (2.2,2.2);
\draw[dashed] (0,-0.8) -- (2.5,1.7);
\draw (1.5,-0.2) node{$\pi=0$};
\draw (1.2,1.8) node{$\pi=0$};
\draw (-1,0.2) node{$\mathcal S_{\alpha_1}(c_1\delta)$};
\draw (2.7,2) node{$\mathcal S_{\alpha_2}(c_1\delta)$};
\draw[very thick]  (-0.20,-0.1)  -- (3.5,0.2);
\draw (3.2,0.7) node{$u=0$};
\draw (3.2,-0.2) node{$u>0$};
\draw[->, semithick] (0,0) -- (0,-1);
\draw[->, semithick] (0,0) -- (-1,1);
\draw (-0.2,-0.2) node {$O$};
\draw (-0.35,-1) node {$-\alpha_1'$};
\draw (-1,1.2) node {$-\alpha_2'$};
\draw (2,0.75) node {$P_+$};
\draw (0,0) node{$\bullet$};
\end{tikzpicture}
\caption{\emph{Case-3.2: The case when $1-\tau_i\leq \epsilon\ll1$.}}
\end{figure}
\begin{figure}[h]
\begin{tikzpicture}
\draw[semithick] (0,0)--(2,2) --(3,1)-- (3,0)--(0,0);
\draw[dashed] (-1.5,0) -- (3.5,0);
\draw[dashed] (-0.8,0.5) -- (3.5,0.5);
\draw[dashed] (-1.2,-1.2) -- (2.2,2.2);
\draw[dashed] (0,-0.8) -- (2.5,1.7);
\draw (1.5,-0.2) node{$\pi=0$};
\draw (1.2,1.8) node{$\pi=0$};
\draw (-1,0.2) node{$\mathcal S_{\alpha_1}(c_1\delta)$};
\draw (2.7,2) node{$\mathcal S_{\alpha_2}(c_1\delta)$};
\draw[very thick]  (-0.20,0.7)  -- (2,-1.5);
\draw (2.7,-1.2) node{$u=0$};
\draw (01,-1.2) node{$u>0$};
\draw[->, semithick] (0,0) -- (0,-1);
\draw[->, semithick] (0,0) -- (-1,1);
\draw (-0.2,-0.2) node {$O$};
\draw (-0.35,-1) node {$-\alpha_1'$};
\draw (-1,1.2) node {$-\alpha_2'$};
\draw (2,0.75) node {$P_+$};
\draw (0,0) node{$\bullet$};
\end{tikzpicture}
\caption{\emph{Case-3.2: The case when $\Lambda\in\mathfrak C$.}}
\end{figure}
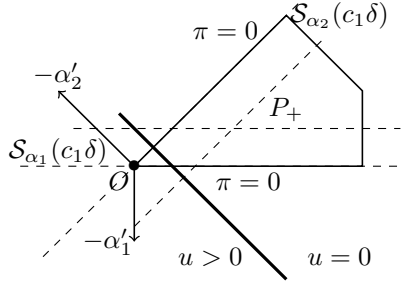

When $\dim(\mathcal V_z)\geq1$, there is at least one facet of $P_+$ in $\mathcal F_1$ that is contained in $\mathfrak F_0(u)$. This implies that $\mathfrak F_+(u)\cap \mathcal F_1$ is contained in at most one facet of $P_+$. We conclude the Lemma using Lemma \ref{ratio-lem} and the above arguments.
\end{proof}

\begin{proof}[Proof of Proposition \ref{ker-thm}]
We will use the arguments of \cite[Section 4]{Wang-Zhou}. According to the structure of the spherical root system $\Phi$, there are three cases: $\dim(\mathcal V_z)=2,1$ or $0$. We will deal with the three cases separately.

\emph{Case-1.} $\dim(\mathcal V_z)=2$. In this case, $\Phi=\emptyset$  and $\mathcal W$ is trivial. With the help of Lemma \ref{small-support-lem}, the proof is almost the same as the toric case in \cite[Theorem 4.1]{Wang-Zhou}. We will omit the details.

\vskip 10pt

\emph{Case-2.} $\dim(\mathcal V_z)=1$. In this case, $\Phi$ contains only one simple roots $\alpha$. The little Weyl group contains only one simple reflection
$$s_{\alpha}(y)=y-2\frac{\langle{\alpha},y\rangle}{|{\alpha}|^2}{\alpha}$$
and the dominate Weyl chamber
$$-\mathcal V=\{y|\langle{\alpha},y\rangle\geq0\}.$$

For a dominate simple piecewise linear function $u$ on $P_+$, set
$$\sigma_u:=\min\{|\text{supp}(u)|_2,|P_+\setminus\text{supp}(u)|_2\}.$$
From Lemma \ref{small-support-lem}, we can conclude that: if $\mathcal L_X(\cdot)>0$ for any non-central dominate convex simple piecewise linear function, then there are constants $\delta_0,\epsilon_0>0$ which depend only on $P$ such that
\begin{align}\label{large-support-dim2}
\mathcal L_X(u)>\epsilon_0>0,
\end{align}
whenever a dominate simple piecewise linear function $u$ has gradient $|\nabla u|=1$ and $$\sigma_u+\langle\nabla u,\alpha\rangle\geq\delta_0.$$

By Proposition \ref{HMA-eq-lem}, the minimizer $u_0$ satisfies the homogenous Monge-Amp\`ere equation. Denote by $W_{\alpha}$ the Weyl wall in $\mathfrak M_\mathbb R$ with respect to $\alpha$. Choose a point $z_O\in\text{Int}(P)\cap W_{\alpha}$ and let $\phi$ be a support function of $u_0$ at $z_O$. By the $\mathcal W$-invariance, we can choose $\phi$ to be a central affine function. Replacing $u_0$ by $u_0-\phi_0$ we may assume that $u_0\geq u_0(z_O)=0$.  Note that $u_0$ is not a central affine function. Thus $u_0\not\equiv0$.

Consider
$$\mathcal T=\{y\in P|u_0(y)=0\}.$$
By Lemma \ref{extreme} and the $\mathcal W$-invariance, we have two subcases: $\mathcal T$ contains a part of the Weyl wall $W_{\alpha}$, or  a segment parallel to $\alpha$ with endpoints on $\partial P$.

\emph{Case-2.1.} $\mathcal T$ contains a part of the Weyl wall $W_{\alpha}$. In this case, either $\mathcal T=W_{\alpha}\cap P$ or $\mathcal T$ has an edge
\begin{align}\label{edge-T}
E\subset\{y|\varpi(y)-\lambda=0\}
\end{align}
for some constant $\lambda>0$ and $\varpi$ satisfying $\langle{\alpha},\varpi\rangle>0$.

\begin{figure}[h]
\begin{tikzpicture}
\draw (0,0) node{$\bullet$};
\draw (1,0) node{$\bullet$};
\draw (1.2,0.3) node{${\alpha}$};
\draw (0,-3) node{$W_{\alpha}$};
\draw [semithick] (-1.25,-2.375) -- (1.25,-2.375) -- (2,0) -- (0,1.475) -- (-2,0) -- (-1.25,-2.375);
\draw (-2,-3.5) node{\emph{Case-2.1 (1)}: $E=\mathcal T=W_{\alpha}\cap P$.};
\draw (0.7,-1) node{$\mathcal T(=E)$};
\draw (-0.2,-0.2) node{$z_O$};
\draw [very thick, -latex] (0,0) -- (1,0);
\draw [very thick] (0,1.475) -- (0,-2.375);
\draw [dashed] (0,2) -- (0,-2.8);
\end{tikzpicture}
\begin{tikzpicture}
\fill[color=gray!20] (-1,0.76) -- (-0.5,-2.375) -- (0.5,-2.375) -- (1,0.76) -- (-1,0.76);
\draw (0,0) node{$\bullet$};
\draw (1,0) node{$\bullet$};
\draw (1.2,0.3) node{${\alpha}$};
\draw (0,-3) node{$W_{\alpha}$};
\draw [semithick] (-1.25,-2.375) -- (1.25,-2.375) -- (2,0) -- (0,1.475) -- (-2,0) -- (-1.25,-2.375);
\draw (-2,-3.5) node{\emph{Case-2.1 (2)}: $\mathcal T$ is a polytope.};
\draw (0.2,-1) node{$\mathcal T$};
\draw (1,-1) node{$E$};
\draw (-1.2,-1) node{$s_{\alpha}(E)$};
\draw (-0.2,-0.2) node{$z_O$};
\draw [very thick, -latex] (0,0) -- (1,0);
\draw [very thick] (-1,0.76) -- (-0.5,-2.375) -- (0.5,-2.375) -- (1,0.76) -- (-1,0.76);
\draw [dashed] (0,2) -- (0,-2.8);
\end{tikzpicture}
\caption{\emph{Case-2.1}}
\end{figure}
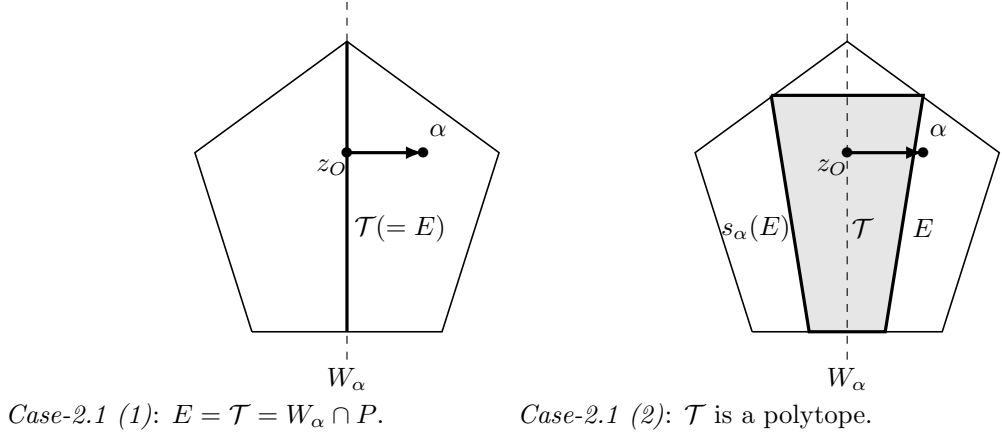

Consider the dominate simple piecewise linear function
$$\bar u(y)=\max\{\varpi(y)-\lambda,0\},~y\in P_+.$$
We will prove that $\mathcal L_X(\bar u)=0$ by using the arguments of \cite[Proof of Theorem 4.1, Step (iii)]{Wang-Zhou}.

Denote
$$\chi_+=\left\{\begin{aligned}1,~&\text{if}~\varpi(y)\geq\lambda,\\0,~&\text{if}~\varpi(y)\leq\lambda.\end{aligned}\right.$$
Let $u_0^+:=u_0\chi_+$ and $u_0^-:=u_0(1-\chi_+)$. Since $u_0\geq0$,  both $u_0^+$ and $u_0^-$ are dominate convex functions on $\mathcal V_+$. Denote
$$a_0:=\inf_{z\in\mathcal T}\lim_{t\to0^+}\frac1t(u_0(z+t{\bar\alpha})-u_0(z)).$$
Clearly, $a_0\geq0.$

\emph{Case-2.1.1.} $a_0>0$. Consider
$$u^\epsilon_+:=\chi_+\{u_0(y)-\epsilon\varpi(y-y_0)\}$$
for a fixed $0<\epsilon<a_0$. 
By the convexity of $P$, $\text{supp}(u^\epsilon_+)=\text{supp}(u_0^+)$ is contained in the strip
$$E+\mathbb R_{\geq0}\varpi:=\{y+t\varpi|t\geq0\}.$$
Note that $u^\epsilon_+$ is convex in $\text{supp}(u^\epsilon_+)$. Hence for any $z\in\text{supp}(u^\epsilon_+)$,
$$u^\epsilon_+(z)=u^\epsilon_+(y+t\varpi)\geq u^\epsilon_+(y)=0,$$
where $z=y+t\varpi$ for some $y\in E$ and $t\geq0$. Hence $u^\epsilon_+$ is convex on $P_+$.

We will further prove that
\begin{align}\label{nabla-alpha-u-eps}
{\alpha}(\nabla u^\epsilon_+)\geq0,
\end{align}
which implies that $u^\epsilon_+$ is $\mathcal W$-dominate.
By our assumption, $E$ is not parallel to $\alpha$. That is, its normal vector $\varpi\not\perp\alpha$. Thus, $E\not=s_{\alpha}(E)$. On the other hand, since $\mathcal T$ is $\mathcal W$-invariant, both $E$ and $s_{\alpha}(E)$ are its edges. Hence the convex hull of $E$ and $s_{\alpha}(E)$ is contained in $\mathcal T$. Then for any fixed $z\in E$, the nonnegative function
$$f_z^\epsilon(t):=u^\epsilon_+(z+t\alpha)$$
vanishes on $t_0\leq t\leq0$ with small $t_0\leq0$. Moreover, $t_0<0$ whenever $z$ in not a vertex of $\mathcal T$ that lies in $W_{\alpha}\cap P$. By convexity, we see that $(f_z^\epsilon)'\geq0$ and this concludes \eqref{nabla-alpha-u-eps}.

Now write
$$\tilde u_0:=u^-_0+u^\epsilon_+=u_0-\epsilon\bar u,$$
which is a dominate convex function. It holds
$$0=\mathcal L_X(u_0)=\mathcal L_X(\tilde u_0)+\epsilon\mathcal L_X(\bar u).$$
But by our assumption, both terms in the right-hand side are nonnegative, we conclude $\mathcal L_X(\bar u)=0$.

\emph{Case-2.1.2.} $a_0=0$. Consider
$$\tilde u^\epsilon_0(y):=u_0^-+\chi_+\max\{0,u_0-\epsilon\bar u(y)\}.$$
Similarly, we can show $\tilde u^\epsilon_0$ is dominate convex.
By \eqref{edge-T}, there is a $\delta(\epsilon)>0$ with $\displaystyle\lim_{\epsilon\to0^+}\delta(\epsilon)=0$
such that
$$\{y|u_0-\epsilon\bar u(y)\leq0\}\subset\{y|0\leq \lambda-\varpi(y)<\delta(\epsilon)\}.$$
By the non-negativity of $\mathcal L_X(\cdot)$, it holds
\begin{align*}
0\leq\epsilon\mathcal L_X(\bar u)&\leq\mathcal L_X(\tilde u^\epsilon_0)+\epsilon\mathcal L_X(\bar u)\\
&=\mathcal L_X(\tilde u^\epsilon_0+\epsilon\bar u)\\
&=\mathcal L_X(u_0)+\mathcal L_X(\tilde u^\epsilon_0-u_0)
\leq C\epsilon\delta(\epsilon)=o(\epsilon),~\epsilon\to0^+,
\end{align*}
for some constant $C$ depending only on $P$ and $u_0$. The last line is true since $\mathcal L_X(u_0)=0$ and
$$\text{supp}(\tilde u^\epsilon_0-u_0)\subset\{y|0\leq\lambda-\varpi(y)<\delta(\epsilon)\}.$$
Hence we get $\mathcal L_X(\bar u)=0$ again.

When $E\perp{\alpha}$, one can choose $\varpi={\alpha}$. We can similarly check that the functions $ u_0$ and $\tilde u_0^\epsilon$ are dominate convex and conclude $\mathcal L_X(\bar u)=0$ as before.

\emph{Case-2.2.} $\mathcal T$ contains a segment $E$ parallel to ${\alpha}$ with endpoints on $\partial P$. In this case, either $\mathcal T=E$ or $E$ is an edge of the polytope $\mathcal T$. In either case, the normal vector $\varpi$ of $E$ satisfies $\varpi\perp{\alpha}$. We can choose a point $y_0\in E$ so that
$$E\subset\{y|\varpi(y-y_0)=0\}~\text{and}~\mathcal T\subset\{y|\varpi(y-y_0)\leq0\}.$$
By similar arguments as \emph{Case-2.1}, we have $\mathcal L_X(\bar u)=0$ for the dominate convex function
$$\bar u(y)=\max\{\varpi(y-y_0),0\},~y\in P_+.$$

\begin{figure}[h]
\begin{tikzpicture}
\draw (0,0) node{$\bullet$};
\draw (1,0) node{$\bullet$};
\draw (1.2,0.3) node{${\alpha}$};
\draw (0,-3) node{$W_{\alpha}$};
\draw [dashed] (0,2) -- (0,-2.8);
\draw [semithick] (-1.25,-2.375) -- (1.25,-2.375) -- (2,0) -- (0,1.475) -- (-2,0) -- (-1.25,-2.375);
\draw (-2,-3.5) node{\emph{Case-2.2 (1)}: $\mathcal T$ is a line segment.};
\draw (1,1) node{$\mathcal T$};
\draw (-0.2,-0.2) node{$z_O$};
\draw [very thick, -latex] (0,0) -- (1,0);
\draw [very thick] (-1,0.76) -- (1,0.76);
\end{tikzpicture}
\begin{tikzpicture}
\fill[color=gray!20] (-1,0.76) -- (-2,0) -- (-1.25,-2.375) -- (1.25,-2.375) -- (2,0) -- (1,0.76) -- (-1,0.76);
\draw (0,0) node{$\bullet$};
\draw (1,0) node{$\bullet$};
\draw (1.2,0.3) node{${\alpha}$};
\draw (0,-3) node{$W_{\alpha}$};
\draw [dashed] (0,2) -- (0,-2.8);
\draw [semithick] (-1.25,-2.375) -- (1.25,-2.375) -- (2,0) -- (0,1.475) -- (-2,0) -- (-1.25,-2.375);
\draw (-2,-3.5) node{\emph{Case-2.2 (2)}: $\mathcal T$ is a polytope.};
\draw (1,-1) node{$\mathcal T$};
\draw (-0.2,-0.2) node{$z_O$};
\draw [very thick, -latex] (0,0) -- (1,0);
\draw [very thick] (-1,0.76) -- (-2,0) -- (-1.25,-2.375) -- (1.25,-2.375) -- (2,0) -- (1,0.76) -- (-1,0.76);
\end{tikzpicture}
\caption{\emph{Case-2.2}}
\end{figure}

\vskip 10pt

\emph{Case-3.} $\dim(\mathcal V_z)=0$. That is, $\Phi_+$ contains two simple roots
$\alpha_1$, $\alpha_2$ which form a basis of $\mathfrak M_\mathbb R$. The only central affine functions are constant functions. Also by the $\mathcal W$-invariance, $P$ contains the origin $O$.

For a dominate simple piecewise linear function
$$\phi_\lambda(y)=\max\{\Lambda(y)-\lambda,0\},~y\in P,$$
we have $\Lambda\in \mathcal V_+$ and  $\Lambda(y)\geq0$ for $y\in P_+$. Hence $$\phi_\lambda(y)=\phi_0(y),~\forall\lambda\leq0.$$
Let $v_{P}$ be the support function of $P$. Then
$$\phi_\lambda(y)=0,~\forall\lambda\geq v_{P}(\Lambda).$$
Thus combining with Lemma \ref{small-support-lem}, we can conclude that: if $\mathcal L_X(\cdot)>0$ for any non-constant dominate convex simple piecewise linear function, then there is are constants $\delta_0,\epsilon_0>0$ which depend only on $P$ such that
\begin{align}\label{large-support-dim2}
\mathcal L_X(u)>\epsilon_0>0,
\end{align}
for any dominate simple piecewise linear function $u$ with $|\nabla u|=1$ and $|\text{supp}(u)|_2\geq\delta_0$.

Note that the only fixed point of the $\mathcal W$-invariant is the origin $O$. Hence $u_0\geq u_0(O)$ if $u_0$ is
 dominate convex. We assume $u_0\geq u_0(O)=0$ by adding a constant.
Since $\mathcal V_z=\{O\}$, we have
$\theta_X=0$ in \eqref{L(u)}.
Again, by Proposition \ref{HMA-eq-lem}, $u_0$ satisfies the homogenous Monge-Amp\`ere equation. By $\mathcal W$-invariance, $\phi_0=u_0(O).$ Without loss of generality, we may assume that $u_0(O)=0$. Consider the set
$$\mathcal T=\{y\in P|u_0(y)=0\}.$$
Clearly $\mathcal T$ is a $\mathcal W$-invariant convex set. By Lemma \ref{extreme}, $\mathcal T$ is either a line segment with endpoints in $\partial P$ or a $\mathcal W$-invariant convex polytope with vertices
on $\partial P$.

\begin{figure}
\begin{tikzpicture}
\draw[semithick] (-1.41,1.41) -- (0,2) -- (1.41,1.41) -- (2,0) -- (1.41,-1.41) -- (0,-2) -- (-1.41,-1.41) -- (-2,0) -- (-1.41,1.41);
\draw [very thick, -latex] (0,0) -- (1,0);
\draw (1,0) node{$\bullet$};
\draw (0,0) node{$\bullet$};
\draw (-0.2,-0.2) node {$O$};
\draw (1.3,0.2) node {$\alpha_1$};
\draw [very thick, -latex] (0,0) -- (0,1);
\draw (0,1) node{$\bullet$};
\draw (-0.3,1.73/2) node {$\alpha_2$};
\draw (0.3,-1.73/2) node {$\mathcal T$};
\draw (0.4,-2.25) node {$W_{\alpha_1}$};
\draw (-2.3,0.3) node {$W_{\alpha_2}$};
\draw[dashed] (-2.4,0) -- (2.4,0);
\draw[dashed] (0,-2.4) -- (0,2.4);
\draw[very thick] (0,-2) -- (0,2);
\draw (-2,-3.5) node{\emph{Case-3.1}: $\mathcal T$ is a line segment.};
\end{tikzpicture}
\begin{tikzpicture}
\fill[color=gray!30] (-1/2,1.73) -- (1/2,1.73) -- (1.25,3*1.73/4) -- (1+1.4/2,1/2) -- (1+1.4/2,-1/2) -- (1.25,-3*1.73/4) -- (1/2,-1.73) -- (-1/2,-1.73) -- (-1.25,-3*1.73/4) -- (-1-1.4/2,-1/2) -- (-1-1.4/2,1/2) -- (-1.25,3*1.73/4) -- (-1/2,1.73);
\draw[thick] (-1,1.73) -- (1,1.73) -- (2,0) -- (1,-1.73) -- (-1,-1.73) -- (-2,0) -- (-1,1.73);
\draw[very thick] (-1/2,1.73) -- (1/2,1.73) -- (1.25,3*1.73/4) -- (1+1.4/2,1/2) -- (1+1.4/2,-1/2) -- (1.25,-3*1.73/4) -- (1/2,-1.73) -- (-1/2,-1.73) -- (-1.25,-3*1.73/4) -- (-1-1.4/2,-1/2) -- (-1-1.4/2,1/2) -- (-1.25,3*1.73/4) -- (-1/2,1.73);
\draw (-0.2,-0.2) node {$O$};
\draw[dashed] (-2.4*1.73/2,1.2) -- (2.4*1.73/2,-1.2);
\draw[dashed] (2.4*1.73/2,1.2) -- (-2.4*1.73/2,-1.2);
\draw[dashed] (0,2.4) -- (0,-2.4);
\draw [very thick, -latex] (0,0) -- (1,0);
\draw (1,0) node{$\bullet$};
\draw (0,0) node{$\bullet$};
\draw (1.3,0) node {$\alpha_1$};
\draw [very thick, -latex] (0,0) -- (-1/2,1.73/2);
\draw (-1/2,1.73/2) node{$\bullet$};
\draw (-1/2-0.3,1.73/2) node {$\alpha_2$};
\draw [very thick, -latex] (0,0) -- (1/2,1.73/2);
\draw (1/2,1.73/2) node{$\bullet$};
\draw (1/2+0.3,1.73/2) node {$\alpha_3$};
\draw (0.3,-1.73/2) node {$\mathcal T$};
\draw (0.8,1.3) node {$E$};
\draw (0.4,-2.25) node {$W_{\alpha_1}$};
\draw (2.2,2/1.73+0.25) node {$W_{\alpha_2}$};
\draw (2.2,-2/1.73+0.25) node {$W_{\alpha_3}$};
\draw (-2,-3.5) node{\emph{Case-3.2}: $\mathcal T$ is a polytope.};
\end{tikzpicture}
\caption{\emph{Case-3}}
\end{figure}

\vskip 10pt

\emph{Case-3.1.} $\mathcal T$ is a line segment. One easily checks that this can happen only if $\Phi_+$ contains precisely two simples roots $\alpha_1\perp\alpha_2$, and $\mathcal T\subset\{y|\langle\alpha_i,y\rangle=0\}$ for one root $\alpha_i$ of them. Without loss of generality we may assume that $\langle\alpha_1,y\rangle=0$ on $\mathcal T$. We will prove
$\mathcal L(\bar u)=0$
for the dominate convex function $\bar u(y)=|\alpha_1(y)|$.

Set
$$a_0:=\inf_{z\in\mathcal T}\lim_{t\to0^+}\frac1t(u_0(z+t\alpha_1)-u_0(z)).$$
Clearly, $a_0\geq0.$

\vskip 10pt

\emph{Case-3.1.1.} $a_0>0$. Then for any $\epsilon<a_0$,
\begin{align}\label{epsilon-str}
\tilde u^\epsilon_0:=u_0-\epsilon\alpha_1(y)
\end{align}
is a dominate convex function. This is true since
\begin{align*}
\alpha_1(\nabla \tilde u^\epsilon_0)(z+t\alpha_1)&\geq \alpha_1(\nabla\tilde u_0^\epsilon)(z)=(a_0-\epsilon)|\alpha_1|^2>0,~\forall t>0,z\in\mathcal T
\end{align*}
and
$$\alpha_2(\nabla \tilde u^\epsilon_0)(y)=\alpha_2(\nabla  u_0)(y)\geq0.$$
Hence, it holds
$$0=\mathcal L_X(u_0)=\mathcal L_X(\tilde u_0^\epsilon)+\epsilon\mathcal L_X(\bar u).$$
However, both terms in the right-hand side are nonnegative. It implies $\mathcal L_X(\bar u)=0$.

\vskip 10pt

\emph{Case-3.1.2.} $a_0=0$. Consider
$$\tilde u^\epsilon_0(y):=\max\{0, u_0-\epsilon\alpha_1(y)\}.$$
We shall first show that $\tilde u^\epsilon_0$ is dominate convex. For each $z\in\mathcal T$ and $t>0$,
consider the function
$$f_z(t):=\tilde u_0(z+t\alpha_1).$$
It is convex on $P_+$ and $f_z(0)=0.$ Suppose that $$t_0=\max\{t\geq0|f_z(t)=0\}.$$
Then by convexity, $\tilde u^\epsilon_0(z+t\alpha_1)\equiv0$ for $0\leq t\leq t_0$.
Hence $f'_z(t_0)>0$. Again by convexity, for any $y=z+t\alpha_1,t>0,$
$$\alpha_1(\nabla\tilde u^\epsilon_0)(y)=f_z'(t)\geq0.$$
Also, since $\alpha_1\perp\alpha_2$, one concludes
$$\alpha_2(\nabla\tilde u^\epsilon_0)(z+t\alpha_1)\geq0,~\forall t\geq0$$
since $u_0$ is $\mathcal W$-dominate convex. In summary we get that $\tilde u^\epsilon_0$ is $\mathcal W$-dominate convex.

Since $\mathcal T\subset W_{\alpha_1}$, there is a $\delta(\epsilon)>0$ with $\displaystyle\lim_{\epsilon\to0^+}\delta(\epsilon)=0$
such that
$$\{y|\tilde u_0^\epsilon(y)=0\}\subset\{y|0\leq\alpha_1(y)<\delta(\epsilon)\}.$$
By the non-negativity of $\mathcal L_X(\cdot)$, it holds
\begin{align*}
0\leq\epsilon\mathcal L_X(\bar u)&\leq\mathcal L_X(\tilde u_0^\epsilon)+\epsilon\mathcal L_X(\bar u)\\
&=\mathcal L_X(\tilde u_0^\epsilon+\epsilon\bar u)\\
&=\mathcal L_X(u_0)+\mathcal L_X(\tilde u_0^\epsilon-u_0)
\leq C\epsilon\delta(\epsilon)=o(\epsilon),~\epsilon\to0^+,
\end{align*}
for some constant $C$ depending only on $P$ and $u_0$. The last line is true since $\mathcal L(u_0)=0$ and
$$\text{supp}(\tilde u_0^\epsilon-u_0)\subset\{y|0\leq\alpha_1(y)<\delta(\epsilon)\}.$$
Hence we get $\mathcal L_X(\bar u)=0$ again.

\vskip 10pt

\emph{Case-3.2.} $\mathcal T$ is a polytope. Clearly $\mathcal T\not=P$. Hence there is an edge $E$ of $\mathcal T$ passing through  a point $y_0\in\text{Int}(P_+)$ with endpoints in $\partial P$. By the convexity of $\mathcal T$, there is a dominate vector $\varpi$ so that
$$\mathcal T\subset\{y|\langle\varpi,y\rangle\leq\langle\varpi,y_0\rangle\}.$$
Let
\begin{align}\label{test-func-322}
\bar u=\max\{\varpi(y-y_0),0\},~y\in P_+.
\end{align}
We will prove that $\mathcal L_X(\bar u)=0$.

Set
$$\chi_+=\left\{\begin{aligned}1,&~\text{if}~\varpi(y)\geq\varpi(y_0)\\0,&~\text{if}~\varpi(y)\leq\varpi(y_0)\end{aligned}\right..$$
Since $u_0\geq0$, $u_0^+:=u_0\chi_+$ and $u_0^-:=u_0(1-\chi_+)$ are dominate convex functions on $P_+$.
Denote $$a_0:=\inf_{z\in E}\lim_{t\to0^+}\frac1t(u_0(z+t\varpi)-u_0(z)).$$
We have two cases as before:

\vskip 10pt

\emph{Case-3.2.1.} $a_0>0$. Consider
$$u^\epsilon_+:=\chi_+\cdot [u_0(y)-\epsilon\varpi(y-y_0)]$$
for a fixed $0<\epsilon<a_0$. It suffices to check that $u^\epsilon_+$ is dominate convex. Once this is known, we can proceed with
$$\tilde u_0:=u^-_0+u^\epsilon_+=u_0-\epsilon\bar u$$
as in \emph{Case-3.1.1}. We will only prove that
\begin{align}\label{nabla-alpha1-u-eps}
\alpha_1(\nabla u^\epsilon_+)\geq0
\end{align}
since the case of $\alpha_2$ can be proved in a same way.

If $E$ is parallel to $\alpha_1$, then $\varpi(\alpha_1)=0$ and $\alpha_1(\nabla u^\epsilon_+)=\alpha_1(\nabla u_0^+)$. Then we get \eqref{nabla-alpha1-u-eps}. If $E$ is not parallel to $\alpha_1$, both $E$ and $s_{\alpha_1}(E)$ are its edges since $\mathcal T$ is $\mathcal W$-invariant. Hence the convex hull of $E$ and $s_{\alpha_1}(E)$ is contained in $\mathcal T$. Then as in \emph{Case-2.1.1}, for any $z\in E$, the nonnegative function
$$f_z^\epsilon(t):=u^\epsilon_+(z+t\alpha_1)$$
vanishes on $t_0\leq t\leq0$ for $t_0\leq 0$ small. Moreover, $t_0<0$ whenever $z$ is not a common vertex of $\mathcal T$ and $P$ that lies in $W_{\alpha_1}$. By convexity, we see that $(f_z^\epsilon)'\geq0$ and  \eqref{nabla-alpha1-u-eps} holds.

\vskip 10pt

\emph{Case-3.2.2.} $a_0=0$. As in \emph{Case-3.1.2}, we consider
$$\tilde u^\epsilon_0(y):=u_0^-+\chi_+\max\{0,u_0-\epsilon\bar u(y)\}.$$
By similar arguments,  we have $\tilde u^\epsilon_0$ is dominate convex and hence $\mathcal L_X(\bar u)=0$.
\end{proof}

\vskip 20pt


\section{K-semistabilty and polystable degenerations of Fano spherical varieties}\label{ap}

One application of Proposition \ref{HMA-eq-lem} is to exclude strictly K-semistable Fano $G$-spherical varieties. Let $M$ be a $\mathbb Q$-Fano $G$-spherical with $P_+$ its moment polytope. By using \cite[Proposition 3.15 and Section 3.2.4]{Del3} (see also \cite[Therems 1.4 and 1.9]{Gagliardi-Hofscheier}), we see that in the Fano case, the coefficients $\lambda_A$'s in \eqref{L(u)} are
$$\lambda_A=1+2\langle\rho,u_A\rangle,~A=1,...,d_+,$$
and consequently, the functional $\mathcal L(\cdot)$ can be simplified as
\begin{align*}
\mathcal L(u)=\int_{P_+}\langle\nabla u,y-2\rho\rangle\pi \,dy,~\forall u\in\mathcal C_{1,\mathcal W}.
\end{align*}

By \cite[Theorem A]{Del3} (see also \cite[Section 3.2]{LZZ}) we have the following criterion:
\begin{lem}\label{sKss-crierion}
$M$ is K-semistable if and only if the functional $\mathcal L(\cdot)$ defined by \eqref{L(u)} is non-negative on $\mathcal C_{1,\mathcal W}$, and is strictly K-semistable if and only if there is in addition a fundamental weight $\varpi\in\mathcal V_+$ with respect to $\Phi_+$ such that $\mathcal L(\ell_\varpi)=0$ for the dominate convex function
\begin{align}\label{u-varpi}
\ell_\varpi(y):=\langle\varpi,y\rangle,~y\in P_+.
\end{align}
\end{lem}

\begin{proof}
The first part that $\mathcal L(\cdot)$ is non-negative is a corollary of Propositions \ref{TC-classification}-\ref{TC-futaki}. It remains to deal with the strictly K-semistable case. For the ``only if" part, by \cite[Theorem A]{Del3}, $M$ is strictly K-semistable if and only if the barycenter of $\mathbf b(P_+)$ against the measure $\pi dy$ satisfies
$$\mathbf{b}(P_+):=\frac{\int_{P_+}y\pi dy}{\int_{P_+}\pi dy}\in\sum_{\alpha\in\Phi_+}\alpha+\partial(\text{Span}_{\mathbb R_{\geq0}}\Phi_+).$$
That is, for the simple roots $\alpha_1,...,\alpha_r\in\Phi_+$, there are non-negative constants $c_1,...,c_r$ so that
$$\mathbf{b}(P_+)-\sum_{\alpha\in\Phi_+}\alpha=\sum_{i=1}^rc_i\alpha_i.$$
As $\mathbf{b}(P_+)$ lies in the boundary, there is at least one simple root $\alpha_{i_0}\in\Phi$ so that $c_{i_0}=0$. The corresponding fundamental weight $\varpi_{i_0}$ given by
$$\varpi_{i_0}(\alpha_{i_0})=\frac12|\alpha_{i_0}|^2~\text{and}~\varpi_{i_0}(\alpha_j)=0,~\forall j\not=i_0.$$
satisfies
$$\varpi_{i_0}(\mathbf{b}(P_+)-\sum_{\alpha\in\Phi_+}\alpha)=0.$$
Note that the function $\ell_{\varpi_{i_0}}(y)$ given by \eqref{u-varpi} is $\mathcal W$-dominate convex since $\varpi_{i_0}$ is a fundamental weight. Hence it defines a $G$-equivariant test configuration $(\tilde M,\tilde L)$ by Proposition \ref{TC-classification}. Since $M$ is K-semistable, by Proposition \ref{TC-futaki}, the Futaki invariant
$$\text{Fut}(\tilde M,\tilde L)=\mathcal L(\ell_{\varpi_{i_0}})=0.$$

Now we turn to the ``if" part. Suppose that $M$ is K-semistable and there is a fundamental weight $\varpi\in\mathcal V_+$ so that $\mathcal L(\ell_\varpi)=0$. By Propositions \ref{TC-classification}-\ref{TC-futaki}, $\ell_\varpi$ defines a $G$-equivariant test configuration with vanishing Futaki invariant. Note that $\varpi\not\in\mathcal V_z$. By the last point of Proposition \ref{TC-classification}, this test configuration is non-product and we get the Lemma.
\end{proof}

Combining with Proposition \ref{HMA-eq-lem}, we can prove:
\begin{prop}\label{ss-limit-KRF}
Let $G/H$ be a spherical homogenous space with irreducible spherical root system $\Phi$. Assume in addition that the valuation cone has trivial central part $\mathcal V_z=\{O\}$. Then there is no strictly K-semistable $\mathbb Q$-Fano compactification of $G/H$ whose moment polytope $P_+$ extends to a $\mathcal W$-invariant convex polytope $P\subset\mathfrak M_\mathbb R$.
\end{prop}

\begin{proof}
Suppose that $M$ is a strictly K-semistable $\mathbb Q$-Fano compactification of $G/H$ with $P_+$ its associated polytope. By Lemma \ref{sKss-crierion}, $\mathcal L(u)\geq0$ for any $u\in\mathcal C_{1,\mathcal W}$ and there is a fundamental weight $\varpi$ such that $\mathcal L(u_\varpi)=0$ for the function defined by \eqref{u-varpi}. By Proposition \ref{HMA-eq-lem}, $u_\varpi$ satisfies the homogeneous Monge-Amp\`ere equation on whole polytope $$P=\displaystyle\bigcup_{w\in\mathcal W}w(P_+).$$ Hence the dimension of the normal mapping $\mathcal N_{u_\varpi}(O)$ of $u_\varpi$ at the origin $O$ can not exceed $\text{rank}(G/H)-1$. Note that each linear piece has normal vector $w(\varpi)$ for some $w\in\mathcal W$ and $\mathcal N_{u_\varpi}(O)$ contains the convex hull Conv$\{w(\varpi)|w\in \mathcal W\}$. Since $\mathcal V_z=\{O\}$, this contradicts to the assumption that $\Phi$ is irreducible.
\end{proof}

We can apply Proposition \ref{ss-limit-KRF} to exclude strictly K-semistable $\mathbb Q$-Fano compactifications of semisimple rank $2$ groups. That is:
\begin{prop}\label{non-existence-sKss-Q-fano}
Let $\mathcal G$ be a semisimple group of rank $2$. Then there is no strictly K-semistable $\mathbb Q$-Fano $\mathcal G$-compactifications.
\end{prop}

\begin{proof}
Recall Example \ref{grp-exa}. The spherical root system $\Phi$ of $\mathcal G$ can be identified with $\Phi^\mathcal G$ and $\mathcal V_z=\{O\}$ since $\mathcal G$ is semisimple. Also by Proposition \ref{moment-polytope-convex}, the moment polytope of any $\mathcal G$-compactification always extends to a $\mathcal W$-invariant polytope. Hence when $\Phi^\mathcal G$ is irreducible, the proposition follows from Proposition \ref{ss-limit-KRF}.

It remains to deal with the cases when $\Phi^\mathcal G$ is of type $A_1\oplus A_1$. This is the only possible case if $\Phi^\mathcal G$ is reducible. In this case, the only possible choices of $\mathcal G$ are $SO_4(\mathbb C)$, $SL_2(\mathbb C)\oplus SL_2(\mathbb C)$, $PSL_2(\mathbb C)\oplus PSL_2(\mathbb C)$ and $PSL_2(\mathbb C)\oplus SL_2(\mathbb C)$. It is proved in \cite{LL-IMRN-2021} that for each fixed semisimple $\mathcal G$, the possible semistable  $\mathbb Q$-Fano $\mathcal G$-compactifications are finite. By running a \texttt{Wolfram Mathematica} programma in \cite[Section 5.2]{LL-IMRN-2021} we confirm that there is no strictly semistable compactifications in these cases.
\end{proof}

Now we turn to the case when rank$(G/H)=2$ and $\dim(\mathcal V_z)=1$, we can show that:
\begin{prop}\label{limit-kss-Q-fano}
Let $G/H$ be a spherical homogenous space of rank $2$ with $\dim(\mathcal V_z)=1$. Let $M$ be a strictly K-semistable $\mathbb Q$-Fano compactification of $G/H$ whose moment polytope $P_+$ extends to a $\mathcal W$-invariant convex polytope $P\subset\mathfrak M_\mathbb R$. Then
\begin{align}\label{ker-L=eq}
\{u\in {\mathcal C_{1,\mathcal W}}|\mathcal L(u)=0\}=\{\ell_\Lambda+c|\Lambda\in\mathcal V_+,c\in\mathbb R\}.
\end{align}
And there is a unique polystable degeneration  that degenerates $M$ to a K-stable $\mathbb Q$-Fano horospherical variety.
\end{prop}

\begin{proof}
By our assumption, $\Phi$ contains only one pair of roots $\pm\alpha$. We can choose coordinates $x,y$ on $\mathfrak a$ so that
$\alpha=(a,0)$ for some $a>0$ that lies on the $Ox$-axis. Suppose that $P_+$ is the moment polytope of $M$. We first prove \eqref{ker-L=eq}.

Suppose that there is some $\bar u\in\mathcal C_{1,\mathcal W}$ so that $\mathcal L(\bar u)=0$. Up to adding a central affine function we may assume that $\bar u$ is non-negative and $\bar u(O)=0$. Then $\mathcal S_0:=\{\bar u=0\}$ is the contact set of $\bar u$ at $O$. By Lemma \ref{extreme}, we have three cases: $\mathcal S_0$ contains an edge orthogonal to $W_\alpha$, $\mathcal S_0$ is a polytope which has an edge totally lies in $P_+$ with endpoints on $\partial P$, or $\mathcal S_0=P\cap W_\alpha$.

\emph{Case-1.} The set $\mathcal S_0$ contains an edge $E$ orthogonal to $W_\alpha$. Suppose that
$$E\subset\{y-\lambda_0=0\}.$$
From the proof of Proposition \ref{ker-thm}, we see that $\mathcal L (\bar u')=0$ for
\begin{align*}
\bar u'=\max\{y-\lambda_0,0\}(\not\equiv0).
\end{align*}
Denote by $v_P(\cdot)$ the support function of $P$. We see that $-v_{P}((0,-1))\leq \lambda_0\leq v_{P}((0,1))$.

Consider
$$f(\lambda):=\int_{P_+\cap\{y-\lambda\geq0\}}yax^2\,dx\wedge dy,~-v_{P}((0,-1))\leq \lambda\leq v_{P}((0,1)).$$
Then
\begin{align*}
f'(\lambda)&=-\int_{P_+\cap\{y-\lambda=0\}}yax^2 \,d\sigma_0=-\lambda\left(a\int_{P_+\cap\{y-\lambda=0\}}x^2 \,d\sigma_0\right).
\end{align*}
Thus $f'(\lambda)\neq 0$ unless $\lambda=-v_{P}((0,-1)),0$ or $v_P((0,1))$. Since $f(-v_{P}((0,-1)))=f(v_{P}((0,1)))=0$ we get $\mathcal L (\bar u')=f(\lambda_0)>0$, a contradiction.

\emph{Case-2.} The set $\mathcal S_0$ is a polytope which has an edge $E$ totally lies in $P_+$. From the proof of Proposition \ref{ker-thm}, there is a dominate convex, simple piecewise linear function
\begin{align}\label{u-test-func}
\bar u'=\max\{l_{\mathbf u},0\}\not\equiv0
\end{align}
where $l_{\mathbf u}(x,y)=u_1x+u_2y-\lambda_0$ with $\mathbf u=(u_1,u_2),~u_1\geq0$ so that
\begin{align}\label{L-u=0}
\mathcal L(\bar u')=\int_{P_+\cap\{u_1x+u_2y-\lambda_0\geq0\}}(u_1(x-a)+u_2y)ax^2\,dx\wedge dy=0
\end{align}
and the crease of $\bar u'$ equals to $E$. Then $\lambda_0>0$. Consider
$$f(\lambda):=\int_{P_+\cap\{u_1x+u_2y-\lambda\geq0\}}(u_1(x-a)+u_2y)ax^2\,dx\wedge dy,~0\leq \lambda\leq v_{P}(\mathbf u).$$
Here, $v_P(\cdot)$ is the support function of $P$. Since $\alpha\in\text{Int}(P_+)$, we have $v_P(\mathbf u)>u_1a$.

By direct computation,
\begin{align*}
f'(\lambda)&=-|\mathbf u|\int_{P_+\cap\{u_1x+u_2y-\lambda=0\}}(u_1(x-a)+u_2y)ax^2 \,d\sigma_0\\
&=-(\lambda-u_1a)\left(a|\mathbf u|\int_{P_+\cap\{u_1x+u_2y-\lambda=0\}}x^2 \,d\sigma_0\right).
\end{align*}
Obviously, we have
\begin{align}\label{f'}
&f(0),f'(0)\geq0,~f'(au_1)=f'(v_P(\mathbf{u}))=0,\notag\\
&f'(\lambda)>0~\text{when}~0<\lambda<u_1a~\text{and}~f'(\lambda)<0~\text{when}~u_1a<\lambda< v_P(\mathbf{u}).
\end{align}
Also note that $f(v_P(\mathbf{u}))=0$. By our assumption, we have $\lambda_0\in(0,v_{P}(\mathbf{u}))$ such that $f(\lambda_0)=0$. However, by \eqref{f'}, $f(\lambda)>0$ for any $\lambda \in(0, v_P(\mathbf{u}))$. We obtain a contradiction again.

\emph{Case-3.}
The set $\mathcal S_0=P\cap W_\alpha$. Thus $\bar u|_{P\cap W_\alpha}=0$. From \emph{Case-2.2} in the proof of Proposition \ref{ker-thm} we see that $\mathcal L(\ell_{\mathbf{u_0}})=0$ for $\mathbf{u_0}=(1,0)$. Consider a point $z_0$ in the interior of $P_+$. Let $\phi_0$ be a support function of $\bar u$ at $z_0$ and $\mathcal T_0$ be the contact set $\{\bar u=\phi_0\}$. Then $\nabla\phi_0\in\mathcal V_+$ and $\mathcal T_0\subsetneq P_+$. Otherwise $\bar u=\phi_0$ and one directly gets \eqref{ker-L=eq}.

We have two subcases:

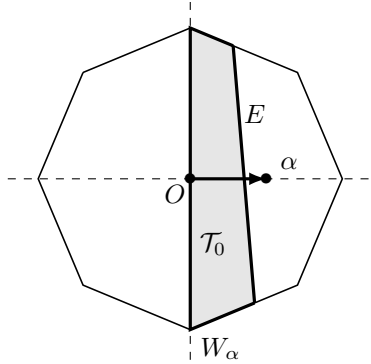
\begin{figure}[h]
\begin{tikzpicture}
\fill[color=gray!20]  (0.846,-0.846-0.8) -- (0,-2) -- (0,2) -- (0.564,1.2+0.564);
\draw[semithick] (-1.41,1.41) -- (0,2) -- (1.41,1.41) -- (2,0) -- (1.41,-1.41) -- (0,-2) -- (-1.41,-1.41) -- (-2,0) -- (-1.41,1.41);
\draw[very thick]  (0.564,1.2+0.564)  -- (0.846,-0.846-0.8);
\draw [very thick, -latex] (0,0) -- (1,0);
\draw (1,0) node{$\bullet$};
\draw (0,0) node{$\bullet$};
\draw (-0.2,-0.2) node {$O$};
\draw (1.3,0.2) node {$\alpha$};
\draw (0.3,-1.73/2) node {$\mathcal T_0$};
\draw (0.85,+1.73/2) node {$E$};
\draw (0.4,-2.25) node {$W_{\alpha}$};
\draw[dashed] (-2.4,0) -- (2.4,0);
\draw[dashed] (0,-2.4) -- (0,2.4);
\draw[very thick]  (0.846,-0.846-0.8) -- (0,-2) -- (0,2) -- (0.564,1.2+0.564);
\end{tikzpicture}
\caption{\emph{Case-3.1}}
\end{figure}

\vskip 10pt

\emph{Case-3.1.} There is a point $z_0$ so that $\mathcal T_0$ intersects an interior point of $P\cap W_\alpha$. Then $\phi_0$ is affine and nonpositive along $P\cap W_\alpha$, attaining its maximum $0$ at an interior point. Hence $\phi_0|_{P\cap W_\alpha}=0$ and $\nabla \phi_0=c\alpha$ for some constant $c>0$. Since $z_0\not\in W_\alpha$, we conclude that $\mathcal T_0$ is the intersection of some half-space with $P_+$, which has an edge $E\subset P_+$ with vertices on $\partial P$. In particular, $E$ does not intersect $W_\alpha$. Suppose that
\begin{align}\label{edge-eq}
E\subset\{(x,y)|xv_1+yv_2-\lambda_0=0\}
\end{align}
for some $\mathbf v=(v_1,v_2)\in\mathcal V_+$ and $\lambda_0>0$. Using the arguments of Propostion \ref{ker-thm}, we conclude that
$\mathcal L(\bar v)=0$  for
$$\bar v=\max\{xv_1+yv_2-\lambda_0,0\}.$$
This is impossible as showed in \emph{Case-2}.

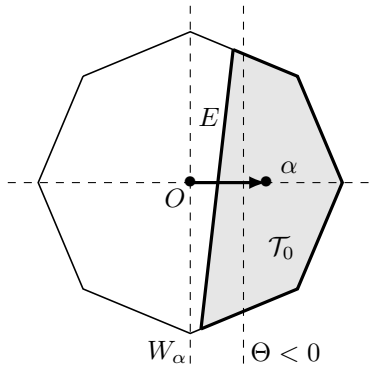
\begin{figure}[h]
\begin{tikzpicture}
\fill[color=gray!20] (0.564,1.2+0.564) -- (1.41,1.41) -- (2,0) -- (1.41,-1.41) -- (0.141,-0.141-1.8);
\draw[semithick] (-1.41,1.41) -- (0,2) -- (1.41,1.41) -- (2,0) -- (1.41,-1.41) -- (0,-2) -- (-1.41,-1.41) -- (-2,0) -- (-1.41,1.41);
\draw[very thick]  (0.564,1.2+0.564)  -- (0.141,-0.141-1.8);
\draw [very thick, -latex] (0,0) -- (1,0);
\draw (1,0) node{$\bullet$};
\draw (0,0) node{$\bullet$};
\draw (-0.2,-0.2) node {$O$};
\draw (1.3,0.2) node {$\alpha$};
\draw (1.2,-1.73/2) node {$\mathcal T_0$};
\draw (0.25,+1.73/2) node {$E$};
\draw (-0.3,-2.25) node {$W_{\alpha}$};
\draw (+1.25,-2.25) node {$\Theta<0$};
\draw[dashed] (-2.4,0) -- (2.4,0);
\draw[dashed] (0,-2.4) -- (0,2.4);
\draw[dashed] (0.7,-2.4) -- (0.7,2.4);
\draw[very thick]  (0.564,1.2+0.564) -- (1.41,1.41) -- (2,0) -- (1.41,-1.41) -- (0.141,-0.141-1.8);
\end{tikzpicture}
\caption{\emph{Case-3.2}}
\end{figure}

\vskip 10pt

\emph{{Case-3.2}}. $\mathcal T_0$ does not pass an interior point of $P_+\cap W_\alpha$ for any $z_0$. Suppose that $E$ is an edge of $\mathcal T_0$ satisfies \eqref{edge-eq} so that $$\mathcal T_0\subset\{(x,y)|xv_1+yv_2-\lambda_0\geq0\}.$$
Then $P\cap W_\alpha$ lies in the left-hand side of $E$.

Recall that when $M$ is K-semistable, the extremal vector field $X=0$, and the function $\Theta$ in \eqref{L-Theta} can be reduced to
$$\Theta(x,y)=\bar S-\frac{a}{|x|}.$$
In this case,
$$\{(x,y)|\Theta<0\}=\{(x,y)|-\frac1{\bar S}a<x<\frac1{\bar S}a\}.$$
By our assumption, we may choose $z_0$ sufficiently close to $W_\alpha$ so that $E\subset\{(x,y)|\Theta<0\}$. Let
$$\tilde u(x,y)=\left\{\begin{aligned}&\phi_0(x,y),&~\text{if}~xv_1+yv_2-\lambda_0\leq0,\\&\bar u(x,y),&~\text{if}~xv_1+yv_2-\lambda_0\geq0.\end{aligned}\right.$$
Then $\tilde u\leq\bar u$ and $\text{supp}(\bar u-\tilde u)\subset\{\Theta<0\}$. Consequently, we have
$$\mathcal L(\tilde u)<\mathcal L(\bar u)=0.$$
A contradiction. Hence it must hold $\bar u=\ell_\mathbf{u}$ for some $\mathbf{u}\in\mathcal V_+$.

Combining with the fact that $\mathcal L(\cdot)$ vanishes on every central affine function, we get \eqref{ker-L=eq}.

\vskip10pt

Finally, we show that any polystable degeneration is induced by $\ell_{\mathbf {u_0}}$ for $\mathbf {u_0}=(1,0)$.
Consider an arbitrary rational vector $\Lambda\in\mathcal V_+$ so that $\alpha(\Lambda)>0$ and the equivariant $\mathbb Z$-test configuration $\mathcal F_\Lambda$ constructed in \cite[Section 3]{LL-arXiv-2021}. By \cite[Corollary 4.6]{LL-arXiv-2021}, these are the only $\mathbb R$-test configurations of $M$ that has irreducible central fibre. And the ``polystable degeneration" must be of this form.

We can show that $\mathcal F_\Lambda$ coincides with the equivariant $\mathbb Z$-test configuration determined by $\ell_\Lambda$ in \cite[Section 2.4]{AK} or \cite[Section 3]{Del3}.  Indeed, by \cite[Section 4.4]{Del-202009}, $\mathcal F_\Lambda$ is a twist of $\mathcal F_{\Lambda'}$ with
$$\Lambda'=\frac{\alpha(\Lambda)}{|\alpha|^2}\alpha.$$
Up to rescaling, this is just $\mathcal F_{\alpha}$, the equivariant $\mathbb Z$-test configuration determined by $\ell_{\mathbf {u_0}}$.
Combining with \cite[Theorem 3.30-Corollary 3.31]{Del3},  the central fibre $\mathcal M_0$ of $\mathcal F_\alpha$ is a $\mathbb Q$-Fano horospherical variety with moment polytope $P_+$ (cf. \cite[Proposition 5.5]{Del3}). By \eqref{ker-L=eq}, the barycenter of $P_+$,
$$\mathbf{b}(P_+)=\frac{\int_{P_+}y\pi(y)\,dy}{\int_{P_+}\pi(y)\,dy}$$
satisfies
$$\mathcal L(\ell_\Lambda)=\langle\Lambda,\mathbf{b}(P_+)-\alpha\rangle\cdot\int_{P_+}\pi dx\wedge dy,~\forall\Lambda\in\mathcal V_+.$$
Hence,
\begin{align}\label{bar-ss-cond}
\mathbf{b}(P_+)=\alpha.
\end{align} Thus $\mathcal M_0$ is K-stable and we get the proposition.
\end{proof}

Recall the optimal test configurations defined in \cite{Gabor-JDG-2008}, a test configuration is called \emph{optimal} if it minimizes
$$\mathcal W(u):=\frac{\mathcal L(u)}{\|u\|_{L^2}}.$$
Clear \eqref{ker-L=eq} shows that:
\begin{cor}
Under the assumption of Proposition \ref{limit-kss-Q-fano}, up to a twisting there is a unique optimal degeneration of $M$ given by $\ell_\alpha$.
\end{cor}

Theorem \ref{Main-thm2} is a combination of Propositions \ref{ss-limit-KRF} and \ref{limit-kss-Q-fano}. Theorem \ref{Main-thm3} follows from \ref{non-existence-sKss-Q-fano}  and \ref{limit-kss-Q-fano} as well as the well-known toric case.

\begin{rem}\label{recover}
The relation \eqref{bar-ss-cond} recovers the semistable criterion of \cite[Theorem 5.3]{Del3} under the assumption of $G/H$ and $M$ of Proposition \ref{limit-kss-Q-fano}.
\end{rem}

\begin{rem}\label{mod}
Let $G/H$ be a spherical homogenous space of rank $2$. Consider the modified K-semistable $\mathbb Q$-Fano compactification $\mathcal X_0$ of a $G/H$. The modified K-stability is defined analogously as Definition \ref{stability-def} according to the modified Futaki invariant (cf. \cite[Section 1]{TZ5}, \cite[Section 4]{BN} and \cite[Sections 1-2]{WZZ-Adv-2016}), and is related to the existence of K\"ahler-Ricci solitons. On a $G$-spherical variety $\mathcal X_0$ with moment polytope $P_+$, the modified Futaki invariant can be reduced to
\begin{align*}
\mathcal L_V(u):=\int_{P_+}\langle\nabla u,y-2\rho\rangle e^{\theta_V(y)}\pi \,dy\geq0,~\forall u\in\mathcal C_{1,\mathcal W}.
\end{align*}
Here $V$ denotes the soliton vector field, which is also a vector in $\mathcal V_z(\mathcal X_0)$, the linear part of the valuation cone of $\mathcal X_0$. If $\mathcal X_0$ is indeed modified K-stable, then it admits a K\"ahler-Ricci soliton. Otherwise, $\mathcal X_0$ is strictly modified K-semistable. In this case, applying Proposition \ref{limit-kss-Q-fano} to $\mathcal L_V(\cdot)$ we see that the ``polystable degeneration" of $\mathcal X_0$ (cf. \cite{Han-Li}) is also unique. In particular in cases of group compactifications, combining with \cite[Theorem 1.3]{LL-arXiv-2021}, one recovers the algebraic uniqueness of the limit of K\"ahler-Ricci flow showed by \cite{Han-Li} in this special case.
\end{rem}

\vskip 20pt


\section{Appendix: The moment polytope}

\begin{prop}\label{moment-polytope-convex}
Let $(M,L)$ be a polarized compactification of the spherical homogenous space $G/H$ and $P_+$ be its moment polytope. Denote by $\mathcal W$ the little Weyl group of $G/H$. Assume that
\begin{itemize}
\item[($\Pi$1)] $L$ has a $G$-invariant divisor, and:
\item[($\Pi$2)] $\Phi^G$ has no simple root that also belongs to the spherical root system $\Phi$;\footnote{When $\alpha\in\Phi_{+,s}^G\cap\Phi$, $\alpha$ is said to be Type-(a) according to the classification of Luna (cf. \cite[Section 30.10]{Timashev-book}).}
\item[($\Pi$3)] Whenever a simple root $\alpha\in\Phi^G_+$ satisfies $\alpha^\vee|_\mathfrak M\not\in\mathbb R\Phi$,
\begin{align*}
\alpha^\vee|_\mathfrak M\in\mathcal V\cup \mathcal V_+^\vee,
\end{align*}
where $\mathcal V_+^\vee$ is the (closed) dual cone of $\mathcal V_+$.
\end{itemize}
Then $P:=\displaystyle\bigcup_{w\in \mathcal W}w(P_+)$ is a convex polytope in $\mathfrak M_{\mathbb R}$.
\end{prop}

\begin{proof}
Denote by $\mathfrak D_{G}$ the set of $G$-invariant prime divisors and $\mathcal D$ the colours of $G/H$ (i.e. the prime $B$-invariant divisors in $G/H$). The map $\varrho(\cdot)$ given by \eqref{varrho} extends to the set $\mathcal D$. Suppose that a $G$-invariant divisor of $L$ is given by
$$\mathfrak d=\sum_{Y\in\mathfrak D_{G}}c_YY.$$
Then according to \cite[Theorem 3.3]{Brion89} (see also \cite[Theorem 3.30]{Timashev-survey}), up to shifting by a vector in $\mathfrak{z^*(g)}\subset\mathcal V^*_z$, the moment polytope $P_+$ consists of all points $y\in\mathfrak M_\mathbb R$ satisfying
\begin{align}
c_Y+u_Y(y)&\geq0,~u_Y\in\mathfrak D_G,\label{polytope-eq-1}\\
\varrho(D)(y)&\geq0,~D\in\mathcal D.\label{polytope-eq-2}
\end{align}
Here each $u_Y\in\mathfrak N\cap\mathcal V$ denotes the prime generator of the ray corresponding to the $G$-invariant prime divisor $Y$. It suffices to show that the polytope defined by \eqref{polytope-eq-1}-\eqref{polytope-eq-2} extends to a $\mathcal W$-invariant one.

Now we recall a classification result of spherical system of Luna (cf. \cite[Section 30]{Timashev-book}). Denote by $\alpha^\vee$ the coroot of a simple root $\alpha\in \Phi^G_+$. For each $D\in\mathcal D$, we can associate to it a simple root $\alpha\in\Phi^G_+$ so that the corresponding minimal parabolic group $P_\alpha\subset G$ moves $D$. In this case we write $D\in\mathcal D(\alpha)$. Then by the assumption ($\Pi$2), there is no $\alpha\in\Phi^G_+$ that lies in the spherical root system $\Phi$ of $G/H$. By \cite[Lemma 30.20]{Timashev-book}, for any $D\in\mathcal D$, it holds\footnote{When \eqref{type-2a} holds, $\alpha$ is said to be Type-(2a).}
\begin{align}\label{type-2a}
\varrho(D)=\frac12\alpha^\vee|_{\mathfrak M},~\text{ if $D\in\mathcal D(\alpha)$ and }2\alpha\in\Phi,
\end{align}
or\footnote{When \eqref{type-b} holds, $\alpha$ is said to be Type-(b).}
\begin{align}\label{type-b}
\varrho(D)=\alpha^\vee|_{\mathfrak M},~\text{ if $D\in\mathcal D(\alpha)$ and }\alpha,2\alpha\not\in\Phi.
\end{align}
It is known that although a colour $D$ may belongs to different $\mathcal D(\alpha)$'s, the quantity $\alpha^\vee|_{\mathfrak M}$ is well-defined. Furthermore, $(\alpha^\vee|_\mathfrak M)^\vee\in\Phi\cup2\Phi$ if $\alpha^\vee|_\mathfrak M\in\mathbb R\Phi$, and any simple root   in $\Phi$  can be derived in this way. We also note that each $u_Y\in\mathcal V=-\mathcal V_+$, which is the inner normal vector of a facet of $P_+$. Hence the set $\hat P_+$ defined by \eqref{polytope-eq-1} and
\begin{align}
\varrho(D)(y)&\geq0,~D\in\mathcal D(\alpha)~\text{such that}~\alpha^\vee|_\mathfrak M\in\mathbb R\Phi\label{polytope-eq-2-1}
\end{align}
lies in the dominate Weyl chamber $\mathcal V_+$ and extends to a $\mathcal W$-invariant convex set $\hat P:=\displaystyle\bigcup_{w\in\mathcal W}w(\hat P_+)$.

Now we consider the remaining simple roots $\alpha\in\Phi^G_+$ satisfying $\alpha^\vee|_\mathfrak M\not\in\mathbb R\Phi$. By ($\Pi$3), if $\alpha^\vee|_\mathfrak M\in\mathcal V_+^\vee$, then
$$\hat P_+\cap\{\varrho(D)(y)\geq0\}=\hat P_+,~D\in\mathcal D(\alpha).$$
If $\alpha^\vee|_\mathfrak M\in\mathcal V$, then $v_D(y):=\max_{w\in\mathcal W}\{-\varrho(D)(w\cdot y)\}$ is a $\mathcal W$-dominate convex function and the polytope
$$\hat P_+\cap\{\varrho(D)(y)\geq0\},~D\in\mathcal D(\alpha)$$
extends to a convex polytope
$$\hat P\cap\{y|v_D(y)\leq0\}.$$

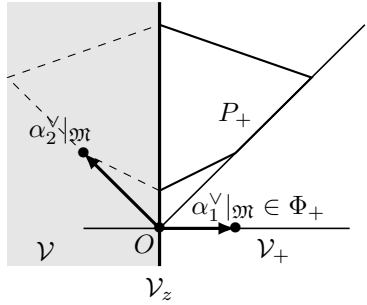
\begin{figure}[h]
\begin{tikzpicture}
\fill[color=gray!20] (-2,-0.5) -- (0,-0.5) -- (0, 3) -- (-2, 3);
\draw [very thick, -latex] (0,0) -- (1,0);
\draw [very thick, -latex] (0,0) -- (-1,1);
\draw (-1,1) node{$\bullet$};
\draw (1,0) node{$\bullet$};
\draw (0,0) node{$\bullet$};
\draw (-0.2,-0.2) node {$O$};
\draw (1.3,0.28) node {$\alpha_1^\vee|_\mathfrak M\in\Phi_+$};
\draw (-1.3,1.3) node {$\alpha_2^\vee|_\mathfrak M$};
\draw (1,1.5) node {$P_{+}$};
\draw (1.5,-0.3) node {$\mathcal V_+$};
\draw (-1.5,-0.3) node {$\mathcal V$};
\draw (0,-0.8) node {$\mathcal V_z$};
\draw[semithick]  (-1,0) -- (2.5,0);
\draw[very thick]  (0,-0.5) -- (0,3);
\draw[thick] (0,0.5) -- (1,1) -- (2,2) -- (0,2.7) -- (0,0);
\draw[dashed] (0,0.5) -- (-1,1) -- (-2,2) -- (0,2.7) -- (0,0);
\draw[semithick] (0,0)-- (2.7,2.7);

\end{tikzpicture}
\caption{\emph{In this example, there is only one simple root $\alpha_1^\vee|_\mathfrak M\in\Phi$, where $\alpha_1$ is a simple root of $\Phi^G_+$. The other simple root $\alpha_2\in\Phi^G_+$ satisfies $\alpha_2^\vee|_\mathfrak M\not\in\Phi$ but $\alpha_2^\vee|_\mathfrak M\in\mathcal V$}.}
\end{figure}

In whichever case, from \eqref{polytope-eq-2} we conclude that $P_+$ extends to a $\mathcal W$-invariant convex polytope.
\end{proof}

\begin{cor}\label{polytope-cor}
Suppose that the spherical homogenous space $G/H$ satisfies the assumptions $(\Pi2)$, $(\Pi3)$ in Proposition \ref{moment-polytope-convex} and
\begin{itemize}
\item[($\Pi$1')] For any $D\in\mathcal D$, there exists a $B$-semi-invariant function $f_D\in\mathbb C[G/H]$ so that $D=\{x\in G/H|f_D(x)=0\}.$
\end{itemize}
Then for any polarized compactification $(M,L)$ of $G/H$, its moment polytope $P_+$ extends to a convex polytope $P:=\displaystyle\bigcup_{w\in \mathcal W}w(P_+)\subset\mathfrak M_{\mathbb R}$. 
\end{cor}

\begin{proof}

By \cite[Proposition 3.1]{Brion89}, there always exists a $B$-invariant $\mathbb Q$-Cartier divisor $\mathfrak d$ of $L$ so that
$$\mathfrak d=\sum_{Y\in\mathfrak D_{G}}c_YY+\sum_{D\in\mathcal D}c_D\overline D.$$
Here $c_y,c_D\in\mathbb Q$ are constants. By assumption ($\Pi$1'), for $f=\prod_{D\in\mathcal D}f_D^{c_D}$ we have
$$\text{div}(f)=\sum_{D\in\mathcal D}c_D D$$
on $G/H$. Hence on $M$,
$$\mathfrak d':=\mathfrak d-\text{div}(f)=\sum_{Y\in\mathfrak D_{G}}c'_YY$$
for some constants $c'_Y\in\mathbb Q$ is a $G$-invariant divisor and we confirm the assumption ($\Pi2$) in Proposition \ref{moment-polytope-convex}. Hence $P_+$ extends to a $\mathcal W$-invariant convex polytope.
\end{proof}

\begin{exa}
By Example \ref{grp-exa}, for the spherical homogenous space $\mathcal{G\times G}/{\rm diag}(G)$, $\Phi_+^\mathcal{G\times G}$ has only simple roots satisfying \eqref{type-b}. Hence $(\Pi2)$ holds. It is also easy to check that there is no simple root $\alpha\in\Phi_+^{\mathcal G\times \mathcal G}$ so that $\alpha^\vee|_\mathfrak M\not\in\mathbb R\Phi$. Hence $(\Pi3)$ holds automatically. Recall that the function $f_D$ for each colour $D$ is constructed in \cite[Section 7]{Timashev-Sbo}, this confirms $(\Pi1')$. Hence the moment polytope of a polarized group compactification extends to a $\mathcal W$-invariant convex polytope.

By a similar argument, we see that Corollary \ref{polytope-cor} also applies to the stable reductive varieties introduced in \cite{AB2,AK}. In particular the central fibre of an equivariant degeneration of a group compactification.
\end{exa}

\begin{exa}
The \emph{space of complete conics} constructed by Brion (cf. \cite[Example 17.12]{Timashev-book}) is a $SL_3(\mathbb C)$-spherical variety. It is a compactification of the spherical homogenous space $$SL_3(\mathbb C)/N_{SL_3(\mathbb C)}(SO_3(\mathbb C)).$$ Denote by $\{\alpha_1,\alpha_2\}$ the two simple roots of $SL_3(\mathbb C)$. The spherical roots are $\{2\alpha_1,2\alpha_2\}$. Thus both $\alpha_1$ and $\alpha_2$ satisfy \eqref{type-2a}.
\end{exa}



\end{document}